\documentclass[a4paper,11pt]{amsart}
\usepackage{amscd,amssymb,amsmath,enumerate}
\usepackage[all]{xy}
\usepackage{color}
\usepackage{graphicx}
\usepackage{graphics}
\usepackage{latexsym}
\usepackage[all]{xy}
\usepackage{psfrag}
\usepackage[T1]{fontenc}
\usepackage{lmodern}
\xyoption{matrix} \xyoption{arrow}

\newtheorem{prop}{Proposition}[section]
\newtheorem{thm}[prop]{Theorem}
\newtheorem{lemma}[prop]{Lemma}
\newtheorem{cor}[prop]{Corollary}

\newtheorem{definition}[prop]{{Definition}}
\newenvironment{defn}{\begin{definition} \rm}{\end{definition}}

\newtheorem{remark}[prop]{{Remark}}

\newcommand{\cA}{{\mathcal A}}

\newcommand{\cX}{{\mathcal X}}
\newcommand{\cZ}{{\mathcal Z}}

\newcommand{\cQ}{{\mathcal Q}}

\newcommand{\cK}{{\mathcal K}}

\newcommand{\cG}{{\mathcal G}}

\newcommand{\br}{{\bf r}}

\renewcommand{\aa}{\mathbb{A}}

\newcommand{\ppq}{\leqslant}
\newcommand{\pgq}{\geqslant}
\newcommand{\as}{\hat}

\newcommand{\Hom}{\operatorname{Hom}\nolimits}
\renewcommand{\Im}{\operatorname{Im}\nolimits}
\newcommand{\Ker}{\operatorname{Ker}\nolimits}
\newcommand{\coker}{\operatorname{coker}\nolimits}
\newcommand{\soc}{\operatorname{Soc}\nolimits}
\newcommand{\rad}{\operatorname{rad}\nolimits}
\newcommand{\Top}{\operatorname{Top}\nolimits}
\newcommand{\Gr}{\operatorname{\mathbf{Gr}}\nolimits}

\newcommand{\Ext}{\operatorname{Ext}\nolimits}

\renewcommand{\dim}{\operatorname{dim}\nolimits}

\newcommand{\val}{\operatorname{val}\nolimits}

\newcommand{\ord}{\operatorname{ord}}

\newcommand{\Id}{\operatorname{Id}\nolimits}

\newcommand{\set}[1]{\left\{ #1\right\}}

\newcommand{\Ag}{\cA_{\Gamma}}
\newcommand{\Agw}{\cA_{\Gamma_W}}
\newcommand{\st}{{\operatorname{\bf{str}}\nolimits}}
\newcommand{\mfo}{\mathfrak o}
\newcommand{\mo}{\mathfrak o}
\newcommand{\mm}{\mathfrak m}
\newcommand{\mq}{\mathfrak q}
\newcommand{\Qg}{{\mathcal Q}_{\Gamma}}
\newcommand{\Ig}{I_{\Gamma}}

\begin{document}

\topmargin 0cm
\oddsidemargin 0.5cm
\evensidemargin 0.5cm
\baselineskip=16pt

\title[The Ext algebra of a Brauer graph algebra]
{The Ext algebra of a Brauer graph algebra}

\author[Green]{Edward L.\ Green}
\address{Edward L.\ Green, Department of
Mathematics\\ Virginia Tech\\ Blacksburg, VA 24061\\
USA}
\email{green@math.vt.edu}
\author[Schroll]{Sibylle Schroll}
\address{Sibylle Schroll\\
Department of Mathematics \\
University of Leicester \\
University Road \\
Leicester LE1 7RH \\
United Kingdom}
\email{ss489@leicester.ac.uk }
\author[Snashall]{Nicole Snashall$^\ast$}\thanks{$^\ast$ Corresponding author}
\address{Nicole Snashall\\
Department of Mathematics \\
University of Leicester \\
University Road \\
Leicester LE1 7RH \\
United Kingdom}
\email{njs5@leicester.ac.uk}
\author[Taillefer]{Rachel Taillefer}
\address{Rachel Taillefer\\ Clermont Universit\'e, Universit\'e Blaise Pascal, Laboratoire de Math\'ematiques\\ BP 10448, F-63000 Clermont-Ferrand -
CNRS, UMR 6620, Laboratoire de Math\'ematiques, F-63177 Aubi\`ere\\
France}
\email{Rachel.Taillefer@math.univ-bpclermont.fr}

\thanks{This work was supported by the Leverhulme Trust in the form of an Early Career
Fellowship for the second author. The third author wishes to thank the University of Leicester for granting study leave. The third and fourth authors would like to thank the Universit\'e
Blaise Pascal (Clermont-Ferrand 2) for a month of invited professorship,
during which part of this collaboration took place.}

\subjclass[2010]{16G20, 
16S37, 
16E05, 
16E30. 
}
\keywords{Brauer graph algebra, Koszul, $d$-Koszul, Ext algebra, finite generation.}

\begin{abstract}
In this paper we study finite generation of the Ext algebra of a Brauer graph algebra by determining the degrees of the generators. As a consequence we characterize the Brauer graph algebras that are Koszul and those that are ${\mathcal K}_2$.
\end{abstract}

\maketitle

\section*{Introduction}

This paper studies finite generation, and the corresponding degrees of the generators, of the Ext algebra of a Brauer graph algebra. We show that if the Brauer graph has no truncated edges then the Ext algebra of the associated Brauer graph algebra is finitely generated in degrees 0, 1 and 2. As a result we characterize those Brauer graph algebras that are ${\mathcal K}_2$ in the sense of Cassidy and Shelton \cite{CS}. Moreover, we determine the Koszul and the $d$-Koszul Brauer graph algebras.

Let $K$ be a field, ${\mathcal Q}$ a finite quiver and $I$ an admissible ideal of $K{\mathcal Q}$. Assume that the algebra $\Lambda = K{\mathcal Q}/I$  is finite dimensional and indecomposable.
Koszul algebras play an important role in representation theory, and it is well-known that if $\Lambda = K{\mathcal Q}/I$ is a Koszul algebra then the Ext algebra $E(\Lambda)$ is finitely generated in degrees 0 and 1. Moreover the ideal $I$ is quadratic, that is, $I$ is generated by homogeneous elements of length 2. This work was motivated by the study of the Koszul Brauer graph algebras, and we determine these algebras in Theorem~\ref{thm:charac_Koszul}, before investigating several generalizations of this concept among Brauer graph algebras.

One such generalization is the class of $d$-Koszul algebras,  where $d \geqslant 2$; these algebras were introduced by Berger in \cite{Ber}, in order to include some cubic Artin-Schelter regular algebras and anti-symmetrizer algebras. Berger, Dubois-Violette and Wambst \cite{BDVW}, and Green, Marcos, Mart\'\i nez-Villa and Zhang \cite{GMMVZ} both continued the study of $d$-Koszul algebras, extending known properties of Koszul algebras. It has since been shown that  some quantum groups and Yang-Mills algebras are included in this framework.

A graded algebra $\Lambda = K{\mathcal Q}/I$ is a $d$-Koszul algebra if the
$n$-th projective module in a minimal graded projective $\Lambda$-resolution of
$\Lambda_0$ (the $K$-space spanned by the vertices of ${\mathcal Q}$) can be generated in degree $\delta(n)$, where the map $\delta:
{\mathbb N} \to {\mathbb N}$ is defined by
\[\delta(n) = \begin{cases}
\frac{n}{2}d & \text{if $n$ is even}\\
\frac{n-1}{2}d + 1 & \text{if $n$ is odd.}
\end{cases}\]
In particular, the ideal $I$ of a $d$-Koszul algebra $K{\mathcal Q}/I$
is generated by homogeneous elements of length $d$. Note that if $d=2$, we
recover the usual (quadratic) Koszul algebras.  The $d$-Koszul Brauer graph algebras are fully determined in Theorem~\ref{thm:charac_Koszul}.

It was shown in \cite{GMMVZ}
that the Ext algebra of a $d$-Koszul algebra is finitely generated in degrees 0,
1 and 2; thus $d$-Koszul algebras are $\cK_2$ algebras in the sense of Cassidy and Shelton \cite{CS}.
They define a graded algebra $A$ to be $\cK_2$ if its Ext algebra $E(A)$ is generated as an algebra in degrees 0, 1 and 2. In Theorem~\ref{thm:fg-ext}, we give a sufficient condition for the Ext algebra of a Brauer graph algebra to be finitely generated in degrees $0, 1$ and $2$. We then characterize the ${\mathcal K}_2$ Brauer graph algebras in Theorem~\ref{thm:notK2}.

Many other generalizations of Koszul algebras have been introduced more recently, such as almost Koszul algebras \cite{BBK}, piecewise Koszul algebras \cite{LHL}, $(p,\lambda)$-Koszul algebras \cite{LZ}, and multi-Koszul algebras \cite{HR}. In this paper, we study a generalization of the Koszul property by Green and Marcos \cite{GM}, where they considered algebras $K{\mathcal Q}/I$ in which the ideal $I$ is generated by homogeneous elements of more than one length. The algebra $K{\mathcal Q}/I$ is 2-$d$-homogeneous if $I$ can be generated by homogeneous elements of lengths 2 and $d$. Green and Marcos then introduced 2-$d$-determined and 2-$d$-Koszul algebras in \cite{GM}. We discuss these algebras in Section~\ref{sec:truncated_edge}, but note here simply that a 2-$d$-Koszul algebra is 2-$d$-determined, which in turn is 2-$d$-homogeneous; moreover a 2-$d$-determined algebra is 2-$d$-Koszul if its Ext algebra is finitely generated. Having determined the 2-$d$-homogeneous Brauer graph algebras in Section~\ref{sec:Koszul}, we return to these algebras in Section~\ref{sec:truncated_edge}, where we give a positive answer for Brauer graph algebras to all three questions posed in \cite{GM}.
As a consequence we are also able to give new classes of 2-$d$-Koszul algebras, which was one of our motivations for this paper. In particular, the second question in \cite{GM} asks whether it is the case that the Ext algebra of a 2-$d$-Koszul algebra of infinite global dimension is necessarily generated in degrees 0, 1 and 2. This is indeed the case for Brauer graph algebras. Moreover, for a 2-$d$-homogeneous Brauer graph algebra $\Ag$ with Brauer graph $\Gamma$, we show in Theorem~\ref{thm:2-d-Koszul}, that the following four conditions are equivalent: (1)~$\Gamma$ has no truncated edges, (2)~$\Ag$ is 2-$d$-determined, (3)~$\Ag$ is 2-$d$-Koszul, and (4)~the Ext algebra of $\Ag$ is generated in degrees 0, 1 and 2 (that is, $\Ag$ is ${\mathcal K}_2$). It should be noted that these properties are not, in general, equivalent, as is demonstrated by Cassidy and Phan in \cite{CP}.

This paper extends the work of Antipov and Generalov, who showed in \cite{AG} that the Ext algebra of a symmetric Brauer graph algebra is finitely generated. We remark that \cite{AG} used different methods, and gave no details of the degrees of the generators.

The body of the paper proves the necessary structural results to describe the uniserial modules, string modules, syzygies and projective resolutions of the simple modules. The paper uses the covering theory for Brauer graphs and Brauer graph algebras developed in \cite{GSS}. We see at the start of Section~\ref{sec:coverings}, that \cite{GSS} enables us to reduce to the case where the Brauer graph has no loops or multiple edges and where the multiplicity function is identically one; this vastly reduces the computations required to determine the Ext algebra. The definition of a Brauer graph algebra also requires a \textit{quantizing function} $\mq$ in order to define some of the relations (see Section~\ref{sec:notation}). If the field is algebraically closed and if either the Brauer graph is a tree or the Brauer graph algebra is symmetric, then we may choose $\mq\equiv 1$. At the end of Section \ref{sec:coverings}, we discuss the quantizing function $\mq$ when the Brauer graph algebra is graded (but not necessarily symmetric or a Brauer tree algebra), and we show that we can assume, given some mild assumptions on the field $K$, that $\mq\equiv 1$ in the study of the Ext algebra.

We would like to thank the referee for their helpful comments.

\section{Background and Notation}\label{sec:notation}

We now introduce Brauer graph algebras, giving the definitions and notation which we use throughout the paper.

Let $\Gamma$ be a finite connected graph with at least one edge.  We denote
by $\Gamma_0$ the set of vertices of $\Gamma$ and by $\Gamma_1$ the
set of edges of $\Gamma$. We equip $\Gamma$ with a {\em multiplicity
function} $\mm \colon \Gamma_0 \rightarrow \mathbb{N} \setminus \{0\}$
and, for each vertex in $\Gamma$, we fix a cyclic ordering $\mathfrak{o}$ of the edges incident with this vertex.
We call the triple $(\Gamma, \mathfrak{o}, \mm)$ a {\em Brauer graph}. We may denote a Brauer graph by $\Gamma$,
where the choice of cyclic ordering and multiplicity function are suppressed.
In all examples a planar embedding of $\Gamma$ is given and we
choose the cyclic ordering to be the clockwise ordering of the edges
around each vertex.

For each $\alpha \in \Gamma_0$, let $\val(\alpha)$
denote the {\em valency} of $\alpha$, that is, the number of edges
incident with $\alpha$ where we count each loop as two edges.
We let $\Gamma_1^{(\alpha)}$ denote the
set of edges of $\Gamma$ which are incident with the vertex $\alpha$.
Let ${\mathcal D}_\mm = \{\alpha \in \Gamma_0 \mid \mm(\alpha)\val(\alpha) = 1\}$.
If $s$ is an edge in $\Gamma$ such that $s \in \Gamma_1^{(\alpha)}$ for some $\alpha \in {\mathcal D}_\mm$ then we
say that $s$ is a {\em truncated edge} at the vertex $\alpha$.
Thus an edge $s$ is truncated at $\alpha$ if and only if $\mm(\alpha) = 1$ and $\val(\alpha) = 1$.
If $s \in \Gamma_1^{(\alpha)}$ for some $\alpha \in \Gamma_0 \setminus {\mathcal D}_\mm$ then
we say that an edge $t$ in $\Gamma$ is the {\em immediate successor} of the edge
$s$ at the vertex $\alpha$ if $t \in \Gamma_1^{(\alpha)}$ and
$t$ directly follows $s$ in the cyclic ordering around $\alpha$.
We denote the immediate successor of $s$ at the vertex $\alpha$ by $s^>$.
In particular, if $s \in \Gamma_1^{(\alpha)}$ where $\val(\alpha) = 1$ and $\mm(\alpha) >1$,
then $s$ is its own immediate successor so $s^> = s$.
An edge $s$ which is a truncated edge at the vertex $\alpha$ has no
immediate successor at the vertex $\alpha$. For ease of notation, from now on we will simply write {\em successor} instead of immediate successor.
We remark that truncated edges are of fundamental importance in determining
the behaviour of the Ext algebra of a Brauer graph algebra.

Following \cite{B} and \cite{K}, we let $K$ be a field and
introduce the Brauer graph algebra of a Brauer graph $\Gamma$.
We associate to $\Gamma$ a quiver $\Qg$ and a
set of relations $\rho_\Gamma$ in the path algebra $K\Qg$, which we call
the {\em Brauer graph relations}, defined below.
Let $\Ig$ be the ideal of $K\Qg$ which is generated by the set $\rho_\Gamma$.
We define the {\em Brauer graph algebra} $\Ag$ of $\Gamma$ to be the quotient $\Ag = K\Qg/\Ig$.
We keep the notation of \cite{GSS} throughout this paper, and now define $\Qg$ and $\rho_\Gamma$.

If the Brauer graph $\Gamma$ is
$\xymatrix{
\alpha \ar@{-}[r] & \beta
}$
with $\mm(\alpha) = \mm(\beta) = 1$ then
$\Qg$ is
$\xymatrix{
\cdot \ar@(dr,ur)[]_x
}$
and $\rho_\Gamma = \{x^2\}$ so
the Brauer graph algebra is $K[x]/(x^2)$.

In the general case, for a Brauer graph $(\Gamma, \mathfrak{o}, \mm)$, which is distinct from
$\xymatrix{
\alpha \ar@{-}[r] & \beta
}$
with $\mm(\alpha) = \mm(\beta) = 1$,
we define the associated quiver $\Qg$  by
\begin{equation*}
\begin{aligned}
(\Qg)_0 & = \{v_s \mid s \in \Gamma_1\}\\
(\Qg)_1 & = \bigcup_{\alpha \in \Gamma_0 \setminus {\mathcal D}_\mm}
\left \{v_s \to v_{s^{>}} \mid s \in \Gamma_1^{(\alpha)}\right\}.
\end{aligned}
\end{equation*}

Now let $\alpha \in \Gamma_0 \setminus {\mathcal D}_\mm$ and let $s \in \Gamma_1^{(\alpha)}$.
In the case where $\val(\alpha) = 1$ and $\mm(\alpha) >1$, we define the successor sequence
of $s$ at the vertex $\alpha$ as the sequence of one element $s = s_0$.
When $\val(\alpha) > 1$, we define the {\em successor sequence of $s$ at $\alpha$} as the
sequence $s_0, s_1, \ldots , s_{\val(\alpha)-1}$ which has the following properties:
\begin{enumerate}
\item[$\bullet$] $s = s_0$
\item[$\bullet$] $s_i = s_{i-1}^>$ \mbox{ for $i = 1, \dots , \val(\alpha)-1$}.
\end{enumerate}
Note that the loops incident with $\alpha$ are listed twice and the other edges precisely once.
Observe also that if $s=s_0,s_1,s_2,\dots,s_{n-1}$ is the successor sequence for $s$ at vertex
$\alpha$, then $s_1,s_2,\dots,s_{n-1},s_0$ is the successor sequence
for $s_1$ at $\alpha$. We set
$s_{\val(\alpha)}=s$, noting that $s$ is the successor of
$s_{\val(\alpha)-1}$ in this case.

In case $\Gamma$ has at least one loop, care must be taken.  In such
circumstances, for each vertex $\alpha$, we choose a distinguished edge,
$s_{\alpha}$, incident with $\alpha$.  If $\ell$ is a loop at $\alpha$,
$\ell$ occurs twice in the successor sequence of $s_{\alpha}$.  We
distinguish the first and second occurrences of $\ell$ in this
sequence and view the two occurrences as two edges in $\Gamma_1$.
Thus, $\Gamma_1$ is the set of all edges with the proviso that loops
are listed twice and have different successors. For example, suppose that the
Brauer graph is
$\Gamma=\xymatrix{\alpha\ar@{-}@(ul,dl)_{s}\ar@{-}@(u,l)_{t}\ar@{-}^{u}[r] & \beta}$
with $\mm(\beta) = 1$ so that $u$ is truncated at $\beta$. Then both $s$ and $t$ occur
twice in the successor sequence of $u$ at $\alpha$; the successor sequence of $u$
at $\alpha$ is $u, s, t, s, t$, and there is no successor sequence of $u$ at $\beta$.
In this case the quiver ${\mathcal Q}_{\Gamma}$ is
$\xymatrix{v_s\ar@/^1pc/[rr]^{}\ar@/^.5pc/[rr]^{} & & v_t\ar@/^.5pc/[ll]^{}\ar@/^.5pc/[dl]\\
& v_u\ar@/^.5pc/[ul] & }$.
Note that there are two arrows in ${\mathcal Q}_{\Gamma}$ from $v_s$ to $v_t$.

In order to define the Brauer graph relations $\rho_\Gamma$ we need
a {\em quantizing function} $\mq$. Let $\cX_\Gamma $ be the set of pairs $ (s,
\alpha)$ such that $\alpha \in \Gamma_0$, $s
\in \Gamma_1^{(\alpha)} $ and $s$ is
not truncated at either of its endpoints, and let $\mq\colon
\cX_\Gamma \to K\setminus\{0\}$ be a set function. We denote $\mq((s,
\alpha))$ by $\mq_{s, \alpha}$. With this additional data we call $(\Gamma,
\mo, \mm, \mq)$ a {\em quantized Brauer graph}.
We remark that if the Brauer graph $\Gamma$ is $\xymatrix{ \alpha \ar@{-}[r] & \beta }$
then $\cX_{\Gamma} = \emptyset$.
If the field is algebraically closed and if either the Brauer graph is a tree or the Brauer graph algebra is
symmetric, then $\mq\equiv 1$ (see \cite{B}).

There are three
types of relations for $(\Gamma, \mo, \mm, \mq)$. Note that we
write our paths from left to right.
For $\alpha \in \Gamma_0 \setminus {\mathcal D}_\mm$ and $s \in \Gamma_1^{(\alpha)}$,
let $s = s_0, s_1, \ldots , s_{\val(\alpha)-1}$ be the successor sequence
of $s$ at $\alpha$ and let $C_{s, \alpha}$ be the cycle in $\Qg$ given by
$$C_{s, \alpha} = a_0a_1 \cdots a_{\val(\alpha)-1}$$
where $a_i$ is the arrow of $\Qg$ given by
$\xymatrix{
v_{s_i}\ar[r]^{a_i} & v_{s_{i+1}}
}$
for $i = 0, \dots , \val(\alpha)-1$.

\textit{Relations of type one.}
Let $\alpha, \beta \in \Gamma_0\setminus {\mathcal D}_\mm$ and
let $s \in \Gamma_1^{(\alpha)}\cap\Gamma_1^{(\beta)}$.
Thus $s$ is not truncated at either $\alpha$ or $\beta$. Then $\rho_\Gamma$ contains
either $\mq_{s, \alpha} C_{s, \alpha}^{\mm(\alpha)} - \mq_{s, \beta} C_{s, \beta}^{\mm(\beta)}$
or $\mq_{s, \beta} C_{s, \beta}^{\mm(\beta)} - \mq_{s, \alpha} C_{s, \alpha}^{\mm(\alpha)}$.
We call this a type one relation. Note that since one of these relations is the negative
of the other, the ideal $\Ig$ does not depend on this choice.

\textit{Relations of type two.} Let $\alpha \in {\mathcal D}_\mm,
\beta \in \Gamma_0\setminus {\mathcal D}_\mm$, and
let $s \in \Gamma_1^{(\alpha)}\cap\Gamma_1^{(\beta)}$.
Thus $s$ is truncated at $\alpha$ but not at $\beta$.
Then $\rho_\Gamma$ contains $C_{s, \beta}^{\mm(\beta)}b_0$ where
$C_{s, \beta} = b_0b_1 \cdots b_{\val(\beta)-1}$.

\textit{Relations of type three.} These relations are
quadratic monomial relations of the form $ab$ in $K\Qg$ where $ab$
is not a subpath of any $C_{s, \alpha}^{\mm(\alpha)}$.

We note that it is well-known that a Brauer graph algebra is special biserial and weakly symmetric.

\bigskip

Throughout this paper, all modules are right modules. We denote the Jacobson radical
$\br_{\Ag}$ of the Brauer graph algebra $\Ag$ by $\br$ when no confusion can arise.
We will use lower case letters such as $s$ and $t$ to denote edges in
$\Gamma$, capital letters $S$ and $T$ to denote
the corresponding simple $\Ag$-modules, and $v_s$ and $v_t$ to denote
the corresponding vertices in $\cQ_{\Gamma}$.
If $S$ is a simple $\cA_{\Gamma}$-module, then we denote the projective $\Ag$-cover of
$S$ by $\pi_S\colon P_S\to S$.

We recall the structure of the indecomposable projective modules from \cite[Section 4.18]{B}. Let $P$ be an indecomposable projective $\Ag$-module corresponding to the vertex $v_s$ in $\cQ_{\Gamma}$ and edge $s$ in $\Gamma$. If the edge $s$ is not truncated, then $P$ has both top and socle isomorphic to $S$, and $\rad P/\soc P$ is a direct sum of two uniserial modules. Let the vertices of $s$ be $\alpha$ and $\beta$ and let $s, s_1, \dots, s_{\val(\alpha)-1}$ be the successor sequence for $s$ at $\alpha$, and $s, t_1, \dots, t_{\val(\beta)-1}$ be the successor sequence for $s$ at $\beta$. Then $\rad P/\soc P \cong U \oplus V$, where $U$ and $V$ have composition series
$$S_1, \dots, S_{\val(\alpha)-1}, S, S_1, \dots, S_{\val(\alpha)-1}, \dots , S, S_1, \dots, S_{\val(\alpha)-1}$$
and
$$T_1, \dots, T_{\val(\beta)-1}, S, T_1, \dots, T_{\val(\beta)-1}, \dots , S, T_1, \dots, T_{\val(\beta)-1}$$
respectively, such that, for $i = 1, \dots, \val(\alpha)-1$, the simple module $S_i$ occurs precisely $\mm(\alpha)$ times and is associated to the edge $s_i$ in $\Gamma$, and, for $j = 1, \dots, \val(\beta)-1$, the simple module $T_j$ occurs precisely $\mm(\beta)$ times and is associated to the edge $t_j$ in $\Gamma$.

In the case where $s$ is truncated, then $P$ is itself uniserial. Suppose that $s$ is not truncated at vertex $\alpha$ and let $s, s_1, \dots, s_{\val(\alpha)-1}$ be the successor sequence for $s$ at  $\alpha$. Then $P$ has composition series $$S, S_1, \dots, S_{\val(\alpha)-1}, S, S_1, \dots, S_{\val(\alpha)-1}, \dots , S, S_1, \dots, S_{\val(\alpha)-1}, S$$ where, for $i = 1, \dots, \val(\alpha)-1$, the simple module $S_i$ occurs precisely $\mm(\alpha)$ times and is associated to the edge $s_i$ in $\Gamma$.

\section{$d$-homogeneous and 2-$d$-homogeneous Brauer graph algebras}\label{sec:Koszul}

Suppose that $\Lambda=K\cQ/I$ where $I$ is a homogeneous ideal with respect to
the length grading. Let $\rho$ be a minimal set of generators for $I$; the
elements in $\rho$ are necessarily homogeneous.
Let $d\pgq2$ and $d'\pgq2$ be distinct integers. We say that $\Lambda$ is \emph{$d$-homogeneous}
(or \emph{quadratic} when $d=2$) if $\rho$ contains homogeneous elements of
length $d$ only. We say that $\Lambda$ is {\em $d$-$d'$-homogeneous} if $\Lambda$ is not $d$-homogeneous or
$d'$-homogeneous and $\rho$ consists of  homogeneous elements of length $d$ or
$d'$. Note that this does not depend on the choice of minimal generating set
$\rho$ for $I.$

In this section we investigate the $d$-homogeneous and $2$-$d$-homogeneous
Brauer graph algebras. Therefore we need to know a minimal generating set for
$\Ig.$ Recall that, for an integer $n\pgq1$, $\aa_n$ is the graph
$\xymatrix{\cdot\ar@{-}[r]&\cdot\ar@{-}[r]&\cdot\cdots\cdot\ar@{-}[r]&\cdot}$
with $n$ vertices and $\widetilde{\aa}_n$ is the circular graph with $n+1$
vertices.

\begin{lemma}\label{lemma:minimal generators BGA}
Let $(\Gamma, \mo, \mm, \mq)$ be a quantized Brauer graph such that $\Gamma$ is
not $\aa_2,$ and let $\Ag=K\Qg/\Ig$
be the associated Brauer graph algebra. Let $\rho\subset \rho_\Gamma$ be a
minimal generating set for $\Ig.$ Then $\rho$ contains all the relations of types
one and three, and it contains the relation of type two associated to the edge
$s$ truncated at $\alpha$ if and only if the successor of $s$ at its other
endpoint $\beta$ is also truncated.
\end{lemma}

\begin{proof} It is clear that relations of type one and three must be in
$\rho$, so we must prove the condition on relations of type two. Since
$\Gamma\neq \aa_2$ and $s$ is truncated at $\alpha,$ the
edge $s$ has a successor $s_1$ distinct from $s$  at $\beta.$ Let
$s=s_0,s_1,\ldots,s_{\val(\beta)-1}$ be the successor sequence of $s$ at $\beta$,
and let $R_2$ be the relation of type two associated to $s$, so that $R_2=C_{s,\beta}^{\mm(\beta)}b_0$
where $C_{s,\beta}=b_0b_1\cdots b_{\val(\beta)-1}$.

First assume that $s_1$ is not truncated at its other endpoint $\gamma$. We want
to prove that the relation $R_2$ is not in $\rho$.
Since $s_1$ is not truncated at either of its endpoints, there is a relation of
type one associated to $s_1$, of the form
$R_1:=\mq_{s_1,\beta}C_{s_1,\beta}^{\mm(\beta)}-\mq_{s_1,\gamma}C_{s_1,\gamma}^{\mm(\gamma)}\in\rho$.
We have $C_{s_1,\beta}=b_1\cdots b_{\val(\beta)-1}b_0$, and we let $C_{s_1,\gamma}=a_0a_1\cdots a_{\val(\gamma)-1}$.
Therefore
$R_2 = C_{s,\beta}^{\mm(\beta)}b_0 = b_0C_{s_1,\beta}^{\mm(\beta)} = \mq_{s_1,\beta}^{-1}b_0R_1 + \mq_{s_1,\beta}^{-1}\mq_{s_1,\gamma}R_3
a_1\cdots a_{\val(\gamma)-1}C_{s_1,\gamma}^{\mm(\gamma)-1}$, where $R_3:=b_0a_0\in\rho$
is a relation of type three. Therefore $R_2\not\in \rho$.

Conversely, assume that $s_1$ is truncated at $\gamma.$ Suppose that
$R_2\not\in\rho$, so that we can write
$R_2=\sum_{i=1}^p\lambda_iR_1^{(i)}\mu_i+\sum_{j=1}^q\lambda_j'R_2^{(j)}\mu_j'+\sum_{k=1}^r\lambda_k''R_3^{(k)}\mu_k''$
for some $\lambda_i,\lambda_j',\lambda_k'',\mu_i,\mu_j',\mu_k''$ in $K\Qg$ and
relations $R_1^{(i)}, R_2^{(j)}, R_3^{(k)}$ of type one, two and three in
$\rho$.
We work in $K\Qg$, which is graded by length.

The relation $R_2$ is monomial, hence must occur in one of the summands. By
definition, the $R_3^{(k)}$ are not subpaths of $R_2$ (since the proper subpaths
of $R_2$ are all subpaths of some $C_{s_i,\beta}$ and $R_2$ has length at least
$3$). Moreover, if $\lambda_j'R_2^{(j)}\mu_j'=R_2$, then $R_2^{(j)}$ is a
product of (some of) the arrows $b_\ell$ for $0\ppq\ell\ppq\val(\beta)-1$, so
that $R_2^{(j)}$ must be a cyclic permutation of $R_2$ and hence equal to $R_2$,
a contradiction. Finally, if $R_2$ is in the term $\lambda_iR_1^{(i)}\mu_i$,
then $R_1^{(i)}$ is a $K$-linear combination of $C^{\mm(\beta)}_{s_t,\beta}$ for some $t$, and another
cycle that does not contain a $b_\ell$,  and, for length reasons, we must have
$t=0$ or $t=1$. But $s_0=s$ and $s_1$ are truncated at $\alpha$ and $\gamma$
respectively, so there are no relations of type one associated to
these edges. We have again a contradiction, and therefore conclude that $R_2$ is in $\rho.$
\end{proof}

We start by describing all $d$-homogeneous Brauer graph algebras for $d\pgq2$.

\begin{prop}\label{prop:quadratic}
\sloppy Let $(\Gamma,\mo,\mm,\mq)$ be a quantized Brauer graph and let $\Ag$
be the  associated Brauer graph
algebra. Then $\Ag$ is quadratic if and only if $(\Gamma,\mo,\mm)$ is one of the following Brauer graphs.
\begin{enumerate}[\bf (1)]
\item $\Gamma=\aa_2$ with $\mm\equiv 1$ and $\mq\equiv 1.$
\item $\Gamma=\aa_n$ with $n\pgq 2$, all multiplicities equal to $1$ except at the
  first and last vertices which are equal to $2$.
\item $\Gamma=\aa_n$ with $n\pgq 3$, all multiplicities equal to $1$ except at
  one end vertex which is equal to $2$.
\item $\Gamma=\aa_n$ with $n\pgq 4$,  $\mm\equiv1$ and $\mq\equiv 1.$
\item $\Gamma=\widetilde{\aa}_n$ with $n\pgq 1$ and $\mm \equiv 1.$
\item $\Gamma=\xymatrix{\cdot\ar@{-}@(ur,dr)^{}}$ with $\mm \equiv 1$.
\end{enumerate}
\end{prop}

\begin{proof}  Suppose that $\Ag$ is quadratic. Let $\rho\subset\rho_\Gamma$ be a minimal generating set for
  $\Ig.$ If $\alpha$ is a vertex in $\Gamma$ such that
$\val(\alpha)>2$, then there is a relation of type one or type two of length at least $3$ in $\rho$. This contradicts the fact that $\Ag$ is quadratic.
Therefore $\val(\alpha)\ppq 2$ for all vertices $\alpha$ in $\Gamma.$

There are two cases to consider.
\begin{enumerate}
\item[(i)] We assume that there is a vertex $\alpha$ in $\Gamma$ with $\val(\alpha)=1.$

Then there is a unique edge $s$ in $\Gamma$ with endpoint $\alpha$. Let $\beta$
denote the other endpoint of $s$.
There are two subcases to consider here.

First suppose that edge $s$ is truncated at $\alpha$,
that is, $\mm(\alpha)=1$.
If  $\val(\beta)=1$, then, since $\Gamma$ is connected, $\Ag=
K[x]/(x^{\mm(\beta)+1})$ which is quadratic if and only if $\mm(\beta)=1$. Thus $\Gamma=\aa_2$ and
$\mm \equiv 1$; this is {\bf(1)}. Note that we can assume that $\mq\equiv 1$
since there are no relations of type one. On the other hand, if $\val(\beta)=2$,  then we have a relation of type two of length
$\val(\beta)\mm(\beta)+1\pgq 3$ in $\Ig$ associated to the edge $s$ incident
with vertex $\beta$. This relation cannot be in $\rho$, so the successor $s_1$ of
$s$ at $\beta$ is not truncated at its other endpoint
$\gamma$, by Lemma \ref{lemma:minimal generators BGA}. Hence we have a relation of
type one associated to $s_1$, of length
$\mm(\beta)\val(\beta)=\mm(\gamma)\val(\gamma)$, in $\rho.$ Therefore
$\mm(\beta)\val(\beta)=\mm(\gamma)\val(\gamma)=2.$ If $\val(\gamma)=1$ then
$\mm(\gamma)=2$ so that $\Gamma=\aa_3$ and the multiplicities are $(1,1,2)$;
this is part of \textbf{(3)}. If $\val(\gamma)=2$, then $\mm(\gamma)=1$ and we
continue, to get $\aa_n$ with $n\pgq 4$ and multiplicities either
$(1,1,\ldots,1,1)$ or $(1,1,\ldots,1,2)$ (the last edge can be truncated if
$n\pgq4$). We have thus obtained \textbf{(3)} and \textbf{(4)}.

We may now assume that there are no truncated edges. Since the
edge $s$ is not truncated at either of its endpoints $\alpha$ and $\beta$, we have a relation of type one associated to $s$ in $\rho$ so that $\val(\beta)\mm(\beta)=2 = \val(\alpha)\mm(\alpha)$.
Since $\val(\alpha) = 1$, we have that $\mm(\alpha) = 2$. Moreover, either
$\val(\beta)=1$ in which case $\Gamma$ is the graph $\aa_2$ with $\mm \equiv 2$, or $\val(\beta)=2$, $\mm(\beta)=1$
and we continue to get the graph $\aa_n$ with multiplicities $(2,1,1,\ldots,1,2)$.  This gives {\bf (2)}.

\item[(ii)] We assume that all vertices have valency $2.$

Let $\alpha$ be a vertex in $\Gamma.$ Since $\val(\alpha)=2,$ there is either a loop, a double edge or two single edges at $\alpha.$

If there is a loop $s$ at $\alpha$, then, since $\Gamma$ is connected, $\Gamma$ is equal to $\xymatrix{\alpha\ar@{-}@(ur,dr)^s}$. Then there is a relation of type one associated to $s$ so that $\val(\alpha)\mm(\alpha)=2$ and therefore $\mm(\alpha)=1$. This is the graph of {\bf (6)}.
If there is a double edge at $\alpha$, then a similar argument shows that $\mm(\alpha)=1.$ If
$\beta$ is the other vertex of this double edge, then we have $\val(\beta)\pgq
2$. However, all vertices in $\Gamma$ have valency at most $2,$  so that
$\val(\beta)=2$. So $\Gamma$ is $\xymatrix{\alpha\ar@{=}[r]&\beta}$ with $\mm \equiv 1$, and we have the graph $\widetilde{\aa}_1$ of {\bf (5)}.
Finally, suppose there are two edges $s$ and $t$ which are incident with $\alpha$. By assumption, neither $s$ nor $t$ is truncated, so that there is a relation of type one associated to both $s$ and $t$, and therefore $\val(\alpha)\mm(\alpha)=2$ and $\mm(\alpha)=1.$ Continuing this argument, shows that $\Gamma=\widetilde{\aa}_n$ with $n\pgq 2$ and $\mm \equiv 1$, which is {\bf (5)}.
\end{enumerate}

This gives all possible Brauer graphs $(\Gamma,\mo,\mm,\mq)$.  We now give the associated
Brauer graph algebras, which are all quadratic.
\begin{enumerate}[\bf (1)]
\item $\Ag=K[x]/(x^2).$

\item \def\objectstyle{\scriptstyle}
$\Ag = K\Qg/\Ig$ where $ \Qg$ is the quiver
  \[\xymatrix{1\ar@/^.5pc/[r]^{a_1}\ar@(ul,dl)_{b_0} & 2\ar@/^.5pc/[r]^{a_2}\ar@/^.5pc/[l]^{\bar{a}_1} & 3\ar@/^.5pc/[l]^{\bar{a}_2}&\cdots&n-2\ar@/^.5pc/[r]^{a_{n-2}}&n-1\ar@/^.5pc/[l]^{\bar{a}_{n-2}}\ar@(ur,dr)^{b_{n-1}}
  } \]
  and the ideal $\Ig$ is generated by $a_i\bar{a}_i-\bar{a}_{i-1}a_{i-1}$, $a_{i-1}a_{i}$ and $\bar{a}_i\bar{a}_{i-1}$ for $2\ppq i \ppq n-2$,
  $a_1\bar{a}_1-b_0^2$, $\bar{a}_{n-2}a_{n-2}-qb_{n-1}^2$, $b_0a_1,$ $\bar{a}_1b_0,$
  $a_{n-2}b_{n-1}$ and $b_{n-1}\bar{a}_{n-2}$ for some nonzero $q\in K$. Note that we
  have scaled the arrows in the quiver so that the quantizing function $\mq$ is 1
  except for one value which we have denoted by $q$; moreover, if $q$ has
  a square root in $K$, then we can replace $q$ by $1$ (see \cite{ES}).

\item $\Ag = K\Qg/\Ig$ where $ \Qg$ is the quiver
  \[\xymatrix{1\ar@/^.5pc/[r]^{a_1} & 2\ar@/^.5pc/[r]^{a_2}\ar@/^.5pc/[l]^{\bar{a}_1} & 3\ar@/^.5pc/[l]^{\bar{a}_2}&\cdots&n-2\ar@/^.5pc/[r]^{a_{n-2}}&n-1\ar@/^.5pc/[l]^{\bar{a}_{n-2}}\ar@(ur,dr)^{b_{n-1}}
  } \]
  and the ideal $\Ig$ is generated by $a_i\bar{a}_i-\bar{a}_{i-1}a_{i-1}$, $a_{i-1}a_{i}$ and $\bar{a}_i\bar{a}_{i-1}$ for $2\ppq i \ppq n-2$,
   $\bar{a}_{n-2}a_{n-2}-qb_{n-1}^2$,
  $a_{n-2}b_{n-1}$ and $b_{n-1}\bar{a}_{n-2}$ for some nonzero $q\in K$, which can be
  replaced by $1$ if $q$ has
  a square root in $K$.

\item $\Ag = K\Qg/\Ig$ where $ \Qg$ is the quiver
  \[\xymatrix{1\ar@/^.5pc/[r]^{a_1} & 2\ar@/^.5pc/[r]^{a_2}\ar@/^.5pc/[l]^{\bar{a}_1} & 3\ar@/^.5pc/[l]^{\bar{a}_2}&\cdots&n-2\ar@/^.5pc/[r]^{a_{n-2}}&n-1\ar@/^.5pc/[l]^{\bar{a}_{n-2}}} \]
  and the ideal $\Ig$ is generated by $a_i\bar{a}_i-\bar{a}_{i-1}a_{i-1}$, $a_{i-1}a_{i}$ and $\bar{a}_i\bar{a}_{i-1}$ for $2\ppq i \ppq n-2$. Note that we
  have scaled the arrows in the quiver so that $\mq\equiv 1$.

\item  $\Ag = K\Qg/\Ig$ where $ \Qg$ is the quiver with $n$ vertices and $2n$ arrows
\[\xymatrix@R=0.3pc@C=0.4pc@M=0pc{
&&&&&&&&&&&\cdot\ar@/^.5pc/[rrrrrd]^{a}\ar@/^.5pc/[llllld]^{\bar{a}}\\
&&&&&&\cdot\ar@/^.5pc/[rrrrru]^{a}\ar@/^.5pc/[llldd]^{\bar{a}}&&&&&&&&&&\cdot\ar@/^.5pc/[rrrdd]^{a}\ar@/^.5pc/[lllllu]^{\bar{a}}\\\\
&&&\cdot\ar@/^.5pc/[rrruu]^{a}\ar@{.}@/_.3pc/[ldd]&&&&&&&&&&&&&&&&\cdot\ar@/^.5pc/[llluu]^{\bar{a}}\ar@{.}@/^.3pc/[rdd]\\\\
&&&&&&&&&&&&&&&&&&&&&&&\\
\\
\\
\\\\\\\\\\\\\\\\
&&&&&&&&&&&&&&&&\\
&&&&&&&&&&&\cdot\ar@/_.3pc/@{.}[rrrrru]\ar@/^.3pc/@{.}[lllllu] }
\]
and the ideal $\Ig$ is generated by  $a_ia_{i+1}$, $\bar{a}_{i-1}\bar{a}_{i-2}$ and $a_i\bar{a}_i-q_i\bar{a}_{i-1}a_{i-1}$,
for $i=0, \ldots , m-1$, with $q_i\in K,$ $q_i\neq 0,$ where the subscripts are taken modulo $n$ and where $a_i$ denotes the arrow that goes from vertex $i$ to vertex
$i+1$ and  $\bar{a}_i$ denotes the arrow that goes from vertex $i+1$
to vertex $i$. Again we can rescale the arrows so that $\mq$ is $1$ at all
but one place (see \cite{EGST}).

\item  $\Ag=K\left[\xymatrix{\cdot\ar@(ur,dr)^\alpha\ar@(ul,dl)_\beta}\right]/(\alpha\beta-q\beta\alpha,
  \alpha^2, \beta^2)$ for some nonzero $q\in K.$
\qedhere\end{enumerate}
\end{proof}

We now turn to $d$-homogeneous algebras with $d\pgq 3.$

\begin{prop}\label{prop:d-homogeneous}
Let $(\Gamma,\mo,\mm,\mq)$ be a quantized Brauer graph and let $\Ag$ be the associated Brauer
graph algebra. Then $\Ag$ is $d$-homogeneous
with $d\pgq 3$ if and only if $\Gamma$ is a star with $n$ edges, for some $n\pgq 1$, such that $n$
divides $d-1$, the multiplicity of the central vertex is $\frac{d-1}{n}$ and the
other multiplicities are equal to $1.$  The algebra $\Ag$ is uniquely determined
by $(\Gamma,\mo,\mm)$; it is the symmetric Nakayama algebra whose quiver is a cycle of length $n$ and its ideal $\Ig$ is generated by all paths of length $d.$
\end{prop}

\begin{proof} If $\Ag$ is $d$-homogeneous, then there are no relations of type
  three, so the quiver $\Qg$ cannot contain
  distinct cycles at the same vertex. In terms of the graph $\Gamma$, this means that all edges in
  $\Gamma$ are truncated at exactly one vertex. Therefore $\Gamma$ is a
  star in which all the outer vertices have multiplicity $1.$ Let $n$ be the
  number of edges in $\Gamma$ and $m$ the multiplicity of the central
  vertex. The only relations in the algebra $\Ag$ are of type two and are of length $nm+1$.
  Hence $nm=d-1.$ Finally, since $\Ag$ is monomial, we may assume that
  $\mq\equiv 1.$
\end{proof}

It is well-known that all the Brauer graph algebras in
Proposition~\ref{prop:quadratic}{\bf (1)}, {\bf (2)}, {\bf (5)}, {\bf (6)} and in Proposition~\ref{prop:d-homogeneous} are $d$-Koszul (see for instance \cite{EGST,ES,ST} and the references therein). However, it is easy to verify that the algebras of Proposition~\ref{prop:quadratic}{\bf (3)} and {\bf (4)} are not Koszul, since the resolution of the simple module at the vertex 1 is not linear in either of these cases. This gives the following result.

\begin{thm}\label{thm:charac_Koszul}  Let $(\Gamma,\mo,\mm,\mq)$ be a quantized Brauer graph and let $\Ag$ be the associated Brauer graph algebra. Then $\Ag$ is Koszul if and only if it is quadratic
and either $\Gamma = \aa_2$ or $\Gamma$ has no truncated edges. For $d\pgq 3$, the Brauer graph algebra $\Ag$ is $d$-Koszul if and only if it is $d$-homogeneous.
\end{thm}

Now, fix an integer $d>2$. We describe the 2-$d$-homogeneous Brauer graph algebras.

\begin{prop}\label{prop:2dhomogeneous} Let $(\Gamma,\mo,\mm,\mq)$ be a quantized Brauer graph and let $\Ag$ be the associated Brauer graph algebra. Then $\Ag$ is 2-$d$-homogeneous if and only if $(\Gamma,\mo,\mm,\mq)$  satisfies one of the following  conditions.
\begin{enumerate}[\bf (1)]
\item For all vertices $\alpha$ in $\Gamma,$ we have $\mm(\alpha)\val(\alpha)=d$.
\item $\Gamma$ has  a truncated edge, $\Gamma\neq\aa_2,$  no two
  successors are truncated, and for every vertex $\alpha$ in $\Gamma$ we have $\mm(\alpha)\val(\alpha)\in\set{1,d}$.
\end{enumerate}
\end{prop}

\begin{proof} Note that there must be at least one edge in $\Gamma$ that is not truncated at
  either of its endpoints,
  otherwise we are in the situation of Proposition~\ref{prop:quadratic}{\bf (1)} or of
  Proposition~\ref{prop:d-homogeneous}, and the algebra $\Ag$ is quadratic or
  $d$-homogeneous. Let $\rho\subset\rho_\Gamma$ be a minimal set of generators
  for $\Ig.$ The proof has two cases.
\begin{enumerate}[(i)]
\item First assume that there is an  edge $s$ in $\Gamma$
that is truncated at the vertex $\alpha$ in $\Gamma.$ Let $\beta$ be
the other endpoint of $s.$ If $\val(\beta)=1$, then we only have relations of type two in $\rho$,
all of the same length, so that $\Ag$ is homogeneous, which gives a
contradiction.

We may therefore assume that there is an edge $t$ in $\Gamma$,
incident with $\beta$  and such that $t\neq s.$  There is a relation of
type two associated to $s$ at $\beta$ of length $\val(\beta)\mm(\beta)+1\pgq 3$. If $t$ is the
successor of $s$ at $\beta$ and if $t$ is  truncated at its other endpoint, then,
by Lemma~\ref{lemma:minimal generators BGA}, this relation of type two is in $\rho,$ so that
$\val(\beta)\mm(\beta)+1=d.$ However, $\Ag$ is not $d$-homogeneous so there must be a
nontruncated edge $u$ incident with $\beta$ so that $\val(\beta)\pgq3.$  Thus
there is a relation of type one associated to $u$ of length
$\val(\beta)\mm(\beta)=d-1$ so that we have $d-1=2$, since $\Ag$ is $2$-$d$-homogeneous. But $\val(\beta)\pgq 3$
so that we have a contradiction. This shows that no two successors are truncated
and none of the relations of type two are in $\rho$.

Therefore the successor $t$ of $s$ at $\beta$ is not truncated, and  there is a relation of type one associated to $t$ of length
$\val(\beta)\mm(\beta)=\val(\gamma)\mm(\gamma)$, where $\gamma$ is the other
endpoint of $t.$ Since $\Ag$ is not quadratic, we must have $\val(\beta)\mm(\beta)=\val(\gamma)\mm(\gamma)=d$.
Continuing in this way, we see that every relation of type one must have length
$d$ and we get {\bf (2)}. Moreover, if \textbf{(2)} is satisfied, then there are
(quadratic) relations of type three since there are at least two
adjacent cycles $C_{t,\beta}$ and $C_{t,\gamma}$ in $\Qg.$ Since there are no relations of type two in $\rho$ and all relations of
type one are of length $d$, it follows that $\Ag$ is
indeed 2-$d$-homogeneous.

\item Now assume that there are no truncated edges in $\Gamma.$ Therefore there
are no relations of type two.

We suppose first that all vertices have valency at least $2$. Then the relations of type one have
   length $2$ or $d$. More precisely, for any edge $s$ with endpoints $\alpha$
   and $\beta$, we must have
   $\val(\alpha)\mm(\alpha)=\val(\beta)\mm(\beta)\in\set{2,d}$. Since $\Gamma$ is
   connected and $\Ag$ is not
   quadratic, we must have $\val(\alpha)\mm(\alpha)=d$ for all vertices $\alpha$, and we are in case {\bf (1)}.
   Moreover, if all vertices $\alpha$ have valency at least $2$ and
   $\val(\alpha)\mm(\alpha)=d$, then there are relations of type three so that $\Ag$ is
   2-$d$-homogeneous.

Finally, we  consider the case where there is a vertex $\alpha$ with $\val(\alpha)=1.$  Let $s$ be the
edge incident with $\alpha$ and let $\beta$ be the other endpoint of $s$. Since $s$ is not truncated at either endpoint, we have $\mm(\alpha)>1$ and $\val(\beta)\mm(\beta)>1.$
If $\val(\beta)=1$, then there are quadratic relations of type three, and a relation of type one associated to $s$ of length $\mm(\alpha)=\mm(\beta)$, and so $\mm(\alpha)=\mm(\beta)$ must equal $d.$ Thus $\Gamma=\aa_2$ with
multiplicity $d$ at each vertex. It is easy to see that the corresponding algebra is 2-$d$-homogeneous.
On the other hand, if $\val(\beta)>1$, let $t$ be another edge incident with $\beta$. By assumption, the edge $t$ is not truncated at its other
endpoint $\gamma.$ Then there are quadratic relations of type three, a relation of type one associated to $s$ of length
$\mm(\alpha)=\val(\beta)\mm(\beta)$ and a relation of type one associated to $t$ of length
$\val(\beta)\mm(\beta)=\val(\gamma)\mm(\gamma)$. Therefore
$\mm(\alpha)=\val(\beta)\mm(\beta)=\val(\gamma)\mm(\gamma)=d$. Continuing in this way,
we have $\val(\varepsilon)\mm(\varepsilon)=d$ at every
vertex $\varepsilon$ in $\Gamma$, which completes {\bf (1)}. The corresponding algebra is 2-$d$-homogeneous.
\qedhere\end{enumerate}
\end{proof}

We end this section with two corollaries which describe in more detail the 2-$d$-homogeneous Brauer graph algebras $\Ag$ in the cases where $\Gamma$ is a star and where $\Gamma$ is $\aa_n$.

\begin{cor} Suppose $\Gamma$ is a star whose central vertex is $\alpha_0$ and
  the other vertices are ordered cyclically $\alpha_1,\ldots,\alpha_n$; set $\alpha_{n+1}=\alpha_1.$ Then the
  associated Brauer graph algebra is 2-$d$-homogeneous if and only if $n$
  divides $d, $ the vertex $\alpha_0$ has multiplicity $\frac{d}{n}$, and for
  every $i$ with $1\ppq i\ppq n$ we have $\mm(\alpha_i)\in\set{1,d}$ and
   $\mm(\alpha_i)\mm(\alpha_{i+1})\in\set{d,d^2}.$
\end{cor}

\begin{cor}
Suppose $\Gamma=\aa_n$. Then the associated generalized Brauer tree algebra is 2-$d$-homogeneous if and only if $n\pgq 3,$ $d$ is even, the multiplicities of the first and last vertex are in $\set{1,d}$ with at least one of them equal to $d$ if $n=3$, and the multiplicities of the other vertices are all equal to $\frac{d}{2}.$
\end{cor}

We now look more generally at the Ext algebra of a Brauer graph algebra. We will return in Section~\ref{sec:truncated_edge} to 2-$d$-homogeneous Brauer graph algebras and the degrees in which the Ext algebra is generated.

\section{The Ext algebra and coverings}\label{sec:coverings}

In this section, we use the covering theory for Brauer graphs which was developed in \cite{GSS} to simplify the calculation of the Ext algebra. We show that we may assume, without any loss of generality, that our quantized Brauer graph $(\Gamma,\mathfrak o,\mm,\mq)$ has $\mm\equiv 1$ and contains no loops or multiple edges. We then discuss the quantizing function $\mq$, proving in Proposition~\ref{prop:reduction of quantizing function}, that if the field $K$ is algebraically closed and if the associated Brauer graph algebra $\Ag$ is length graded, then the number of generators of the Ext algebra $E({\Ag})$ and their degrees do not depend on $\mq.$

For a finite dimensional indecomposable algebra $\Lambda=K{\mathcal Q}/I$ defined by a finite quiver ${\mathcal Q}$ and an admissible ideal $I$, let $J$ be the ideal in $K{\mathcal Q}$ which is generated by the arrows of ${\mathcal Q}$ and let $\br$ denote the Jacobson radical $J/I$ of $\Lambda$. The Ext algebra (or cohomology ring) of $\Lambda$ is given by $$E(\Lambda) = \Ext^*_\Lambda(\Lambda/\br,\Lambda/\br) = \oplus_{n \geqslant0}\Ext^n_\Lambda(\Lambda/\br,\Lambda/\br)$$ with the Yoneda product. If the ideal $I$ is generated by length homogeneous elements, then the length grading of
$K{\mathcal Q}$ induces a grading $\Lambda = \Lambda_0 \oplus \Lambda_1 \oplus
\Lambda_2 \oplus \cdots $, where $\Lambda_0$ is the $K$-space spanned by the
vertices of ${\mathcal Q}$. The graded Jacobson radical of $\Lambda$ is $\br =
\Lambda_1 \oplus \Lambda_2 \oplus \cdots $, and $\Lambda_0 \cong \Lambda/\br$.

Let $(\Gamma,\mathfrak o,\mm,\mq)$ be a quantized Brauer graph and let
$\cA_{\Gamma}$ denote the associated Brauer graph algebra.
We recall the following definitions from \cite{GSS}. For each $\alpha \in \Gamma_0$, we define $\cZ_\alpha$ to be the set $\cZ_\alpha = \{(s,t) \mid s, t \in \Gamma_1, t \mbox{ is the successor of $s$ at vertex $\alpha$}\}.$ Let $\cZ_\Gamma$ be the disjoint union
$$\cZ_\Gamma = \bigcup^{\bullet}_{\alpha \in \Gamma_0}\cZ_\alpha.$$ Let $G$ be a finite abelian group. A set function $W\colon \cZ_\Gamma \to G$ is called a {\em successor weighting} of the Brauer graph $(\Gamma, \mo, \mm,\mq)$. For $\alpha \in \Gamma_0$ we define the {\em order of $\alpha$}, denoted $\ord(\alpha)$, to be the order in $G$ of the element $\omega_\alpha = \prod_{(s,t) \in \cZ_\alpha} W(s,t).$ A successor weighting $W\colon \cZ_{\Gamma}\to G$ of the Brauer graph $(\Gamma, \mo, \mm,\mq)$ is called a {\em Brauer weighting} if $\ord(\alpha)\mid \mm(\alpha)$ for all $\alpha\in\Gamma_0$.

Let $W\colon \cZ_{\Gamma}\to G$ be a
Brauer weighting, and let $\Agw$ be the Brauer graph algebra
associated to the quantized Brauer covering graph $(\Gamma_W,\mathfrak o_W,\mm_W,\mq_W)$.
Let $\br_{\cA_{\Gamma}}$ (respectively, $\br_{\Agw})$ be the
Jacobson radical of $\Ag$ (respectively, $\Agw$). By \cite[Theorem
5.3]{GSS}, $\Agw$ is the covering algebra associated to a weight
function $W^*\colon(\cQ_{\Gamma})_1\to G$. Hence the Ext algebra
$\Ext^*_{\Ag}(\Ag/\br_{\Ag},\Ag/\br_{\Ag})$ is generated in degrees
$d_1,\dots,d_s$ if and only if the Ext algebra
$\Ext^*_{\Agw}(\Agw/\br_{\Agw},\Agw/\br_{\Agw})$ is generated in
degrees $d_1,\dots,d_s$.  In fact, we have the following well-known
result.

\begin{prop}\label{prop:ext-struct} Keeping the notation above, the ring structure of
$\Ext^*_{\Ag}(\Ag/\br_{\Ag},\Ag/\br_{\Ag})$ can be completely
determined from the ring structure of
$\Ext^*_{\Agw}(\Agw/\br_{\Agw},\Agw/\br_{\Agw})$.
\end{prop}

\begin{proof} The $G$-grading on $\Ag$ induced by the weight function
$W^*\colon(\cQ_{\Gamma})_1\to G$ induces a $G$-grading on
$\Ext^*_{\Ag}(\Ag/\br_{\Ag},\Ag/\br_{\Ag})$ such that, if $g\in G$,
and $S$ and $T$ are simple $\Ag$-modules then
\[\Ext^n_{\Ag}(S,T)_g
= \Ext^n_{\Gr}(S,\sigma(g)T),\]
where the right hand side is the graded Ext-group, $S$ and $T$ are viewed as graded
modules with support in degree $1_G$, and $\sigma$ is the shift
functor.  The graded Yoneda product
$\Ext^n_{\Gr}(T,\sigma(g)U)\times \Ext^n_{\Gr}(S,\sigma(h)T)$ is
defined in the usual way after noting that
$\Ext^n_{\Gr}(T,\sigma(g)U)\cong \Ext^n_{\Gr}(\sigma(h)T,
\sigma(h)\sigma(g)U)$.  Finally, using that the category of
$G$-graded $\Ag$-modules is equivalent to the category of
$\Agw$-modules we obtain the desired result.
\end{proof}

Now, in \cite[Theorem 7.7]{GSS}, it was shown, for any quantized Brauer
graph $(\Gamma_0,\mo_0,\mm_0,\mq_0)$, that there is a tower of quantized Brauer covering graphs
$(\Gamma_0,\mo_0,\mm_0,\mq_0), (\Gamma_1,\mo_1,\mm_1,\mq_1)$, $(\Gamma_2,\mo_2,\mm_2,\mq_2), (\Gamma_3,\mo_3,\mm_3,\mq_3)$
such that the topmost quantized Brauer covering graph
$(\Gamma_3,\mo_3,\mm_3,\mq_3)$ has multiplicity function $\mm_3$ identically one, and
the graph $\Gamma_3$ has no loops or multiple edges.

Hence, with Proposition~\ref{prop:ext-struct}, we
may assume that $(\Gamma,\mathfrak o,\mm,\mq)$ is a quantized Brauer graph with $\mm\equiv 1$ and with no loops or multiple edges.

\bigskip

We now consider the quantizing function $\mq$ in the case where $\Ag$ is length graded. It is known that if the field is algebraically closed and if either the Brauer graph is a tree or the Brauer graph algebra is
symmetric, then we can always rescale the arrows so that $\mq\equiv 1$. The next result shows that we may also assume that $\mq\equiv 1$ in the case where $\Ag$ is length graded, since the number of generators of the Ext algebra $E(\Ag)$ and their degrees do not depend on $\mq$.

We begin by introducing some additional notation. If edge $t$ is the successor of edge
$s$ in $\Gamma$ at vertex $\alpha$, we denote the corresponding
arrow in $\cQ_{\Gamma}$ from vertex $v_s$ to vertex $v_t$
by $a(s,t,\alpha)$. In fact, this arrow is uniquely determined by $s$ and $t$.
For, suppose there are vertices $\alpha$ and $\beta$
in $\Gamma$ such that $a(s,t,\alpha)$ and $a(s,t,\beta)$ are arrows.
Since there are no loops in $\Gamma$, we have $s\ne t$. If $\alpha\ne \beta$, then
$s$ and $t$ are distinct edges with
endpoints $\alpha$ and $\beta$, contradicting the assumption that
there are no multiple edges in $\Gamma$. Hence $\alpha = \beta$. So
if edge $t$ is the successor of edge
$s$ at vertex $\alpha$ in $\Gamma$, then the vertex $\alpha$ is
unique. Thus we denote the arrow in $\cQ_{\Gamma}$ from vertex $v_s$
to vertex $v_t$ simply by $a(s,t)$.

\begin{prop}\label{prop:reduction of quantizing function}
Let $K$ be an algebraically closed field and let $(\Gamma,\mo,\mm,\mq)$ be a
quantized Brauer graph with $\mm\equiv 1$ and with
no loops or multiple edges.  Let $\Ag$ be the associated Brauer
graph algebra. Suppose that $\Ag$ is length graded. Then the number of generators of
$\Ext_{\Ag}(\Ag/\br_{\Ag},\Ag/\br_{\Ag})$ and their degrees do not depend on $\mq.$
\end{prop}

\begin{proof}
Let $(\Gamma,\mo,\mm)$ be a Brauer graph with $\mm\equiv 1$ and with no loops or multiple edges. We may assume that $\Gamma$ is not a
star ($\aa_2$ included) since the associated Brauer graph algebras are all
monomial and hence do not depend on a quantizing function $\mq$. Therefore there exists an edge $s$ with endpoints $\alpha$ and $\beta$
such that $v:=\val(\alpha)>1$ and $\val(\beta)>1.$

Let $A$ be the Brauer graph algebra associated to $(\Gamma,\mo,\mm)$ with
quantizing function identically $1$. We shall twist $A$ by a graded algebra automorphism $\sigma$ of $A$ so that
$A^\sigma$ is the Brauer graph algebra associated to the quantized Brauer graph $(\Gamma,\mo,\mm,\mq)$ with
$\mq$ equal to $1$ except at $(s,\alpha)$ and $(s,\beta)$, then use \cite{CS} to see that the Ext
algebras of $A$ and $A^\sigma$ have the same number of generators in the same degrees. This means
that we will have changed precisely one relation in the generating set for $\Ig$,
namely $C_{s,\alpha}-C_{s,\beta}$ will become
$\mq_{s,\alpha}C_{s,\alpha}-\mq_{s,\beta}C_{s,\beta}$ or, to simplify notation,
$rC_{s,\alpha}-C_{s,\beta}$ where $r=\mq_{s,\alpha}\mq_{s,\beta}^{-1}.$ None of the
other relations will change. 

Recall that the product in $A^{\sigma}$ is given by $x\cdot
y=x\sigma^{\ell(x)}(y)$ where  $x$ and $y$ are length homogeneous elements in $A$ and $\ell(x)$ is the length of $x.$

Let $s=s_0,s_1,s_2,\ldots,s_{v-1}$ be the successor sequence of $s$ at $\alpha$,
and let $a_i=a(s_i,s_{i+1})$ be the corresponding arrows in the quiver $\Qg$, for
$i=0,1,\ldots,v-1$ (where $s_v=s$). In this notation, $C_{s,\alpha}=a_0a_1\cdots
a_{v-1}.$ We shall define a graded algebra automorphism $\sigma$ of $A$ by
setting, for $i=0,1,\ldots,v-1$, $\sigma(a_i)=r_ia_i$ for some $r_i\in K\backslash\set{0}$ to be determined, and fixing all
other arrows and vertices in $\Qg$. Suppose we have such an automorphism
$\sigma.$ Then, in
$A^\sigma$, the cycle $C_{s,\alpha}$ becomes
$r_1r_2^2\cdots r_{v-1}^{v-1}a_0a_1\cdots a_{v-1}.$ The arrows
$a_0,\ldots,a_{v-1}$ occur in at most $v$ relations of type one, involving
the cyclic permutations of $C_{s,\alpha}$. Therefore we want
\begin{align*}
&r_1r_2^2\cdots r_{v-2}^{v-2} r_{v-1}^{v-1}a_0a_1\cdots a_{v-2} a_{v-1}=ra_0a_1\cdots a_{v-2} a_{v-1}\\
&r_2r_3^2\cdots r_{v-1}^{v-2}r_0^{v-1}a_1a_2\cdots a_{v-1}a_0=a_1a_2\cdots
a_{v-1}a_0\\ &\vdots\\
&r_0r_1^2\cdots r_{v-3}^{v-2}r_{v-2}^{v-1}a_{v-1}a_0a_1\cdots a_{v-2}=a_{v-1}a_0a_1\cdots
a_{v-2},
\end{align*} so we must solve the system
\begin{align*}
&r_1r_2^2\cdots r_{v-2}^{v-2} r_{v-1}^{v-1}=r\\
&r_2r_3^2\cdots r_{v-1}^{v-2}r_0^{v-1}=1\\ &\vdots\\
&r_0r_1^2\cdots r_{v-3}^{v-2}r_{v-2}^{v-1}=1.
\end{align*}

If $v=2,$ the system is immediately solved: $r_1=r$ and $r_0=1.$ If $v=3$, it is
easy to see that $r_0=r_1^{-2},$ $r_2=r_1^4$ and $r_1^9=r$ so that choosing a
$9$-th root of $r$ for $r_1$ defines $\sigma.$ Now suppose
that $v>3.$
Starting with the last equation, we can express $r_0$ in terms of $r_1, \ldots, r_{v-2}$
and then $r_{v-1}$ also in terms of $r_1, \ldots, r_{v-2}$.
At the next stage, $r_{v-2}$ may be written in terms of $r_1, \ldots, r_{v-3}$,
so that $r_0$ and $r_{v-1}$ may also be written in terms of $r_1, \ldots, r_{v-3}$.
Continuing in this way, we see that
$r_1,\ldots,r_{v-2}$ must be equal up to $v$-th roots of unity so that
\[ r_2=\zeta_2 r_1,\ r_3=\zeta_3r_1,\ \ldots, r_{v-2}=\zeta_{v-2}r_1 \]
for some  $v$-th roots of unity $\zeta_2,\ldots,\zeta_{v-2}.$  We then have
$r_0=\zeta_2^{-3}\cdots\zeta_{v-2}^{-(v-1)}r_1^{-\frac{(v+1)(v-2)}{2}}$ and
$r_{v-1}=\zeta_2^{2}\cdots\zeta_{v-2}^{(v-2)}r_1^{(v+1)(v-2)-\frac{(v+2)(v-3)}{2}}$. Therefore \[r=r_1r_2^2\cdots r_{v-1}^{v-1}=\zeta_2^{2}\cdots\zeta_{v-2}^{(v-2)}(\zeta_2^{2}\cdots\zeta_{v-2}^{(v-2)})^{v-1}r_1^{\varphi(v)}=r_1^{\varphi(v)}\] where
$\varphi(v)=\frac{v^{2}(v-1)}{2}.$ Choosing a $\varphi(v)$-th root of $r$ for
$r_1$ and $\zeta_2=\cdots=\zeta_{v-2}=1$ defines an automorphism $\sigma$ as required.
Note that $\varphi(3)=9$, so that we have defined the same automorphism $\sigma$
in the case $v=3$.

We now use \cite{CS}, where the authors show that the Ext algebra of
$A^\sigma$ is
obtained from the Ext algebra of $A$ by twisting (they consider a
connected graded algebra, but the proof and result are easily adapted to a
quotient of a path algebra by a length homogeneous ideal). Twisting
does not change the number of generators of the Ext algebra or their degrees.

Proceeding in this way for each relation of type one, we see
that the number of generators of the Ext algebra and their degrees do not
depend on $\mq.$
\end{proof}

As a corollary of the proof, we may relax the condition that $K$ is
algebraically closed.

\begin{cor}\label{cor:reduction of quantizing function}
Let $K$ be a field and let $(\Gamma,\mo,\mm,\mq)$ be a
quantized Brauer graph with $\mm\equiv 1$ and with
no loops or multiple edges.  Let  $\Ag$ be the associated Brauer
graph algebra. We assume that one of the following conditions holds:
\begin{enumerate}[(i)]
\item the valency of every vertex in $\Gamma$ is at most two, or
\item there is an integer $k$ such that the valency of every vertex in $\Gamma$ is either $1$ or
  $k$ and the field $K$ contains a root of the polynomial
 $X^{k^2(k-1)/2}-r$  for any  $r\in K.$
\end{enumerate}  Then the number of generators of
$\Ext_{\Ag}(\Ag/\br_{\Ag},\Ag/\br_{\Ag})$ and their degrees do not depend on $\mq.$
\end{cor}

\begin{proof}
In both cases, the ideal $\Ig$ is length homogeneous, and hence $\Ag$ is length graded. It then
follows from the proof of Proposition \ref{prop:reduction of quantizing function} that the result holds.
\end{proof}

From now on, we assume that $\mq \equiv 1$, and write $(\Gamma,\mathfrak o,\mm)$ for a Brauer graph with $\mq \equiv 1$. We assume that $(\Gamma,\mathfrak o,\mm)$ is a Brauer graph with no loops or multiple
edges and $\mm\equiv 1$. The next sections describe the structure of certain classes of modules of a Brauer graph algebra.

\section{Structure of indecomposable modules}

Let $(\Gamma,\mathfrak o,\mm)$ be a Brauer graph with no loops or multiple
edges and $\mm\equiv 1$. Let $\Ag$ denote the associated Brauer graph
algebra. We assume $\mq \equiv 1$.
The following result is used in Sections~\ref{sec:unis} and \ref{sec:string} where we determine the structure of uniserial and string
$\Ag$-modules.

\begin{prop}\label{prop:proj-struct}
Let $(\Gamma,\mathfrak o,\mm)$ with $\mm\equiv 1$ be a Brauer graph
with no loops or multiple edges
and let $\cA_{\Gamma}$ denote the associated Brauer graph algebra.
Let $S$ and $T$ be simple $\Ag$-modules associated to the edges $s$
and $t$ in $\Gamma$.
\begin{enumerate}[\bf (1)]
\item If $S\not\cong T$, then  $\dim_K(\Hom_{\Ag}(P_S,P_T))\leqslant 1$.
\item If $S\cong T$, then $\dim_K(\Hom_{\Ag}(P_S,P_T)) = 2$.
\item We have that $\dim_K(\Hom_{\Ag}(P_S,P_T)) = 1$ if, and only if,
$S\not\cong T$ and there is a uniserial module with top $S$ and
socle $T$. In this case, the uniserial module is unique up to
isomorphism.
\end{enumerate}
\end{prop}

\begin{proof}
\begin{enumerate}[\bf (1)]
\item Suppose that $S\not\cong T$ and assume for contradiction
that $\dim_K(\Hom_{\Ag}(P_S,P_T))\geqslant 2$.  Denote the endpoints of
edge $t$ in $\Gamma$ by $\alpha$ and $\beta$.  By our no loops
assumption, $\alpha\ne \beta$.
Since $\Ag$ is a
special biserial selfinjective algebra, $\rad(P_T)/\soc(P_T)$ is
either a uniserial module $U$ or a direct sum of two uniserial
modules $L_1\oplus L_2$.  Since $\dim_K(\Hom_{\Ag}(P_S,P_T))\geqslant 2$ and $\soc(P_T) = T$,
the simple $S$ occurs at least twice as a composition factor of $\rad(P_T)/\soc(P_T)$. If $S$ occurs as a composition factor of either $U$
or one of the $L_i$'s at least two times, then either $\mm(\alpha)\geqslant 2$,
$\mm(\beta)\geqslant 2$ or $s$ is a loop, which all contradict our
hypothesis. On the other hand, suppose that $S$ occurs as a
composition factor of both $L_1$ and $L_2$.  Then $s$ occurs in the
successor sequences of $t$ at both vertices $\alpha$ and $\beta$.
Hence, $s$ also has endpoints $\alpha$ and $\beta$.  But then $s$
and $t$ are distinct edges between $\alpha$ and $\beta$, which
contradicts the hypothesis that there are no multiple edges.
Thus $\dim_K(\Hom_{\Ag}(P_S,P_T))\leqslant 1$.

\item This is proved by a similar argument to that in {\bf (1)}.

\item First assume that $S\not\cong T$ and there is a uniserial module $V$
having top $S$ and socle $T$.  Then $V$ embeds in the injective
module $P_T$. Hence $S$ must be a composition factor of $P_T$. Then
$\dim_K(\Hom_{\Ag}(P_S,P_T))=1$ by {\bf (1)}. Next suppose that
$\dim_K(\Hom_{\Ag}(P_S,P_T))=1$.  By {\bf (2)}, $S\not\cong T$. If
$P_T$ is uniserial, then it follows that there is a uniserial
submodule of $P_T$ with top $S$ and socle $T$ since $S$ is a
composition factor of $P_T$. Otherwise, we may suppose that $\rad(P_T)/\soc(P_T)=L_1\oplus
L_2$, and, by assumption, $S$ is a composition factor of exactly one of
$L_1$ or $L_2$. So suppose that $S$ is a composition factor of $L_1$ and
$g$ is the composition of the canonical surjections $\rad(P_T)\to \rad(P_T)/T$
and $\rad(P_T)/T\to L_1$.  Let $V=g^{-1}(L_1)$. Then $V$ is a
uniserial submodule of $P_T$ having $S$ as a composition factor.
Hence, there is a uniserial module with top $S$ and socle $T$.

It remains to show that if $V_1$ and  $V_2$ are uniserial modules
with top $S$ and socle $T$, then $V_1\cong V_2$.  Note that $V_1$
and $V_2$ both embed in $P_T$.  If $V_1\not\cong V_2$ then $S$ would
occur at least twice as a composition factor of $P_T$, contradicting
{\bf (1)}. This completes the proof.
\end{enumerate}
\end{proof}

\begin{cor}\label{cor:proj-im} Let $(\Gamma,\mathfrak o,\mm)$ with $\mm\equiv 1$ be a Brauer graph
with no loops or multiple edges
and let $\cA_{\Gamma}$ denote the associated Brauer graph algebra.
Let $S$ and $T$ be simple $\Ag$-modules associated to the edges $s$
and $t$ in $\Gamma$. Then we have the following.
\begin{enumerate}[\bf (1)]
\item If $f,g\colon P_S\to \rad(P_T)$ are nonzero morphisms, then $\Im(f)=\Im(g)$.
\item There are only a finite number of submodules of an indecomposable projective $\Ag$-module.
\item A submodule $M$ of an indecomposable projective $\Ag$-module is determined by the simple
$\Ag$-modules occurring in the top of $M$.
\item There are only a finite number of quotient modules of an
indecomposable projective $\Ag$-module.
\end{enumerate}
\end{cor}

\begin{proof}
We see that {\bf (1)} follows from Proposition~\ref{prop:proj-struct}{\bf (1)} and {\bf (2)}.

To show {\bf (2)} and {\bf (3)}, let $P_T$ be an
indecomposable projective $\Ag$-module and let
$M$ be a nonprojective, nonsimple submodule of $P_T$. Then we have an inclusion
$f\colon M\to P_T$.  If $P_S$ is an indecomposable projective
$\Ag$-module and $g\colon P_S\to M$ is a module homomorphism such
that the induced map $\bar g\colon P_S\to M/\rad M$ is nonzero, then
there is a nonzero map $h = f\circ g\colon P_S\to \rad(P_T)$.  By {\bf (1)}, $\Im(h)$
is unique. Now $P_T$ is either uniserial or biserial.  If $P_T$ is
uniserial, then $M=\Im(h)$ and, by Proposition~\ref{prop:proj-struct}, both
{\bf (2)} and {\bf (3)} follow.

Now suppose that $P_T$ is biserial with $\rad(P_T)/\soc(P_T)=L_1\oplus L_2$.  By Proposition~\ref{prop:proj-struct}{\bf (1)},
we see $\Im(h)/\soc(P_T)\subseteq L_1$ or $\Im(h)/\soc(P_T)\subseteq
L_2$.  Assume, without loss of generality, that $\Im(h)/\soc(P_T)\subseteq L_1$.  If
$M/\soc(P_T)\subseteq L_1$ then $M=\Im(h)$ and there are only a
finite number of such submodules $M$.
So suppose that $M/\soc(P_T)\not\subseteq L_1$
so that $M\ne\Im(h)$. Note that $M/\soc(P_T)\not\subseteq L_2$. Then, since
$\rad(P_T)/\soc(P_T)=L_1\oplus L_2$, we have $M/\rad M\cong S\oplus S'$, for
some simple $\Ag$-module $S'$.  By Proposition~\ref{prop:proj-struct}{\bf (1)} and {\bf (2)}, $S\not\cong
S'$.  Defining $h'\colon P_{S'}\to P_T$ in a similar fashion to the
definition of $h$, we see that $\Im(h')/\soc(P_T)\subseteq L_2$ and
$M/\soc(P_T)=\Im(h)/\soc(P_T)\oplus\Im(h')/\soc(P_T)$.  Parts {\bf (2)}
and {\bf (3)} now follow.

The proof of {\bf (4)} follows from {\bf (2)}.
\end{proof}

Let $S$ be the simple $\Ag$-module associated to the edge $s$ in $\Gamma$. We remarked at the end of Section~\ref{sec:notation} that $P_S$ is uniserial if and only if $s$ is a truncated edge.
The next result is more specific on the structure of the indecomposable projective
$\mathcal{A}_\Gamma$-modules, in the case where  $\mm\equiv 1$ and $\Gamma$ has no loops or
multiple edges; its proof is straightforward and we leave it to the reader.

\begin{lemma}\label{up-low-uniser}
Let $(\Gamma, \mathfrak o,\mm)$ with $\mm\equiv 1$ be a Brauer graph
with no loops or multiple edges and let $\cA_{\Gamma}$ denote the
associated Brauer graph algebra. Assume $S$ is a simple $\Ag$-module
such that $P_S$ is biserial and $U$ is a simple $\Ag$-module such
that $P_U$ is uniserial.  If $s$ is the edge in $\Gamma$ associated
to $S$ and $s$ has endpoints $\alpha$ and $\beta$, then let
$s=s_0,s_1,s_2, \dots,s_{m-1}$ and $s=t_0,t_1,t_2,\dots,t_{n-1}$ be
the successor sequences for $s$ at vertices $\alpha$ and $\beta$
respectively. Let $S_i$ (resp. $T_i$) be the simple $\Ag$-module
associated to the edge $s_i$ (resp. $t_i$).  If $u$ is the edge in
$\Gamma$ associated to $U$ and $u$ has endpoints $\alpha'$ and
$\beta'$ with $u$ truncated at $\beta'$, then let $u=u_0,u_1,u_2,\dots,u_{p-1}$
be the successor sequence for $u$ at $\alpha'$.  Let $U_i$ be the
simple $\Ag$-module associated to the edge $u_i$.
\begin{enumerate}[\bf (1)]
\item The composition factors of $P_S$ are
$\{S,S_1,\dots,S_{m-1},T_1,\dots,T_{n-1},S\}$.
\item The composition factors of $P_U$ are
$\{U,U_1,\dots,U_{p-1},U\}$.
\item For $i=1,\dots,m-1$ and $j=1,\dots,n-1$, $S_i\not\cong T_j$.
\item For $0\leqslant i<j\leqslant m-1$, $S_i\not\cong S_j$.
\item For $0\leqslant i<j\leqslant n-1$, $T_i\not\cong T_j$.
\item For $0\leqslant i<j\leqslant p-1$, $U_i\not\cong U_j$.
\end{enumerate}
\end{lemma}

\section{Structure of uniserial modules}\label{sec:unis}

In this section we describe the structure of the uniserial modules of a Brauer graph algebra.

\begin{prop}\label{uniser-sub} Let $(\Gamma,\mathfrak o,\mm)$ with $\mm\equiv 1$ be a Brauer graph
with no loops or multiple edges and let $\cA_{\Gamma}$ denote the
associated Brauer graph algebra. Assume $S$ is a simple $\Ag$-module
such that $P_S$ is biserial and $U$ is a simple $\Ag$-module such
that $P_U$ is uniserial.  If $s$ is the edge in $\Gamma$ associated
to $S$ and $s$ has endpoints $\alpha$ and $\beta$, then let
$s=s_0,s_1,s_2, \dots,s_{m-1}$ and $s=t_0,t_1,t_2,\dots,t_{n-1}$ be
the successor sequences for $s$ at vertices $\alpha$ and $\beta$
respectively. Let $S_i$ (resp. $T_i$) be the simple $\Ag$-module
associated to the edge $s_i$ (resp. $t_i$) and $S_m=T_n=S$.  If $u$
is the edge in $\Gamma$ associated to $U$ and $u$ has endpoints
$\alpha'$ and $\beta'$ with $u$ truncated at $\beta'$, then let
$u=u_0,u_1,u_2,\dots,u_{p-1}$ be the successor sequence for $u$ at
$\alpha'$.  Let $U_i$ be the simple $\Ag$-module associated to the
edge $u_i$ and $U_p=U_0=U$. Let $L$ be a nonzero uniserial
$\Ag$-module with composition series for $L$, $(0)=L_{k+1}\subset
L_{k}\subset \cdots\subset L_1\subset L_0=L$.

\begin{enumerate}[\bf (1)]
\item If the socle of $L$ is isomorphic to $S$, then
 either $0\leqslant k\leqslant m-1$ and, for $j=0,\dots,k$,
$L_j/L_{j+1}\cong S_{m-k+j}$ or  $0\leqslant k\leqslant n-1$ and, for
$j=0,\dots,k$, $L_j/L_{j+1}\cong T_{n-k+j}$.

\item If the socle of $L$ is isomorphic to $U$, then  $0\leqslant k\leqslant p$
and $j=0,\dots,k$, $L_j/L_{j+1}\cong U_{p-k+j}$.
\end{enumerate}
\end{prop}

\begin{proof} We prove {\bf (1)} and leave {\bf (2)} to the reader.
If $L$ is a simple module, then $L\cong S= S_m$, and
taking $k=0$ we see that {\bf (1)} follows. Now assume that $L$ is a
nonsimple uniserial module with socle $S$. It follows that $L$ is
isomorphic to a submodule of $P_S$ since $P_S$ is the injective
envelope of $S$. Since $\mm\equiv 1$ and $\Gamma$ has no loops or
multiple edges, $\rad(P_S)$ is the sum of two uniserial modules $X$
and $Y$ such that
\begin{enumerate}
\item[(i)] $X\cap Y = \soc(P_S)$,
\item[(ii)] if $0= X_{m}\subset X_{m-1}\subset \cdots\subset
X_{0} = X$ is the composition series for $X$, then, for
$j=0,\dots,m-1$, $X_j/X_{j+1}\cong S_{j+1}$,
\item[(iii)] if $0= Y_{n}\subset Y_{n-1}\subset \cdots\subset
Y_{0} = Y$ is the composition series for $Y$, then, for
$j=0,\dots,n-1$, $Y_j/Y_{j+1}\cong T_{j+1}$.
\end{enumerate}
By Corollary~\ref{cor:proj-im}{\bf (3)} and Lemma~\ref{up-low-uniser}{\bf (5)}, the uniserial module $L$ must be
isomorphic to a submodule of either $X$ or $Y$ and the result
follows.
\end{proof}

An immediate consequence of the above result is the following.

\begin{cor}\label{cor:top-soc-uniser} Let $(\Gamma,\mathfrak o,\mm)$ with $\mm\equiv 1$ be a Brauer graph
with no loops or multiple edges and let $\cA_{\Gamma}$ denote the
associated Brauer graph algebra. If $L$ and $L'$ are two nonsimple,
nonprojective uniserial $\Ag$-modules such that $\soc(L)\cong
\soc(L')$ and $\Top(L)\cong \Top(L')$, then $L\cong L'$.
\end{cor}

\section{Structure of string modules}\label{sec:string}

Let $(\Gamma,\mathfrak o,\mm)$ with $\mm\equiv 1$ be a Brauer graph
with no loops or multiple edges and let $\cA_{\Gamma}$ denote the
associated Brauer graph algebra. We now classify the string modules for $\Ag$ and begin with
uniserial modules.  Let $L$ be a (nonzero) uniserial $\Ag$-module.
There are 3 cases to consider. The first case is that $L$ is a projective-injective
module. There is no special notation for this case. The second case is
that $L$ is isomorphic to a simple $\Ag$-module, $S$. Let $s$ be the
edge in $\Gamma$ associated to $S$. In this case, we denote $L$ (up
to isomorphism) by $\st(s^+)$, where $\st(s^+)\cong S$. The final
case is that $L$ is a nonsimple nonprojective uniserial module with
top $T$ and socle $S$. Let $s$ and $t$ be the edges in $\Gamma$
associated to the simple modules $S$ and $T$ respectively. By
Corollary~\ref{cor:top-soc-uniser}, $T$ and $S$ completely determine $L$
up to isomorphism.  We denote $L$ by either $\st(t^+,s^-)$ or
$\st(s^-,t^+)$.

If $s$ and $t$ are edges in $\Gamma$, and the successor sequence for $s$ at vertex $\alpha$ is
$s=s_0,s_1,\dots,s_{m-1}$, then we say \emph{$t$ occurs in the
successor sequence for $s$ at vertex $\alpha$} if $\alpha$ is a
nontruncated endpoint of $s$ and $t=s_i$, for some $1\leqslant i\leqslant m-1$.
Clearly $t$ occurs in the successor
sequence for $s$ at vertex $\alpha$ if and only if $s$ occurs in the
successor sequence for $t$ at vertex $\alpha$.  Note also that $t$ cannot
occur in the successor sequence
for $s$ at both endpoints $\alpha$ and $\beta$, for otherwise $\alpha=\beta$ which
contradicts the assumption that there are no loops in $\Gamma$. Let
$S$ and $T$ be the simple $\Ag$-modules associated to $s$ and $t$
respectively. By Proposition~\ref{uniser-sub} and its proof, we see
that there is a uniserial module $L$, unique up to isomorphism, with
top $T$ and socle $S$, if and only if $s=t$ or $t$ occurs in the
successor sequence for $s$ at some vertex $\alpha$.

Thus, summarizing the description of uniserial string modules, we have the
projective-injective uniserial modules, the simple modules
$\st(s^+)$, and the modules of the form $\st(t^+,s^-)$ where $t$
occurs in the successor sequence for $s$ at some vertex $\alpha$.

\bigskip

We now describe the nonuniserial string modules for $\Ag$ in terms of
sequences of weighted edges in the Brauer graph $\Gamma$.

\begin{defn}
For $n\geqslant 2$, let $\as{s}_1,\dots,\as{s}_n$ be edges in $\Gamma$. We assign either $+$
or $-$ to each edge $\as{s}_i$ and denote this assignment by either
$\as{s}_i^+$ or $\as{s}_i^-$.  Consider the sequence
$\sigma=\as{s}_1^{e_1},\as{s}_2^{e_2},\dots, \as{s}_n^{e_n}$, where
$e_i\in\{+,-\}$ for $i=1,\dots,n$. We say $\sigma$ is an
\emph{acceptable sequence of weighted edges} if the following hold.
\begin{enumerate}
\item[{\bf (ST1)}] For $i=1,\dots,n-1$, $e_i\ne e_{i+1}$.
\item[{\bf (ST2)}] For $i=1,\dots,n-1$, there is $\alpha_i\in\Gamma_0$ such that $\as{s}_{i}$ and $\as{s}_{i+1}$ belong to $\Gamma_1^{(\alpha_i)}$.
\item[{\bf (ST3)}] For $i=1,\dots,n-1$, $\as{s}_i\neq \as{s}_{i+1}$.
\item[{\bf (ST4)}] For $i=1,\dots,n-2$, $\alpha_i\ne\alpha_{i+1}$.
\end{enumerate}
\end{defn} Conditions \textbf{(ST2)} and \textbf{(ST3)} combined mean that for $i=1,\dots,n-1$, there are vertices $\alpha_i$ in $\Gamma$, such that $\as{s}_{i+1}$ occurs in the successor
sequence for $\as{s}_i$ at vertex $\alpha_i$.

We note that, for $i=1,\dots, n-1$, the vertex $\alpha_i$ is
uniquely determined by $\as{s}_i$ and $\as{s}_{i+1}$ by our assumption that
$\Gamma$ has no loops or multiple edges. We use the notation $\as{s}$ in order to distinguish
acceptable sequences and successor sequences, but it may happen that  $\as{s}_{i+1}$ is in fact the
successor of $\as{s}_i$ at one of its vertices.

The following result is straightforward and the proof is left to the reader.

\begin{lemma}\label{accept-seq}
Let $(\Gamma,\mathfrak o,\mm)$ with $\mm\equiv 1$ be a Brauer graph with no
loops or multiple edges  and suppose that $\sigma= \as{s}_1^{e_1},\as{s}_2^{e_2},\dots,
\as{s}_{n}^{e_{n}}$ is an acceptable sequence of weighted edges in
$\Gamma$ with $n\geqslant 2$. Then
\begin{enumerate}[\bf (1)]
\item $\as{s}_n^{e_n},\as{s}_{n-1}^{e_{n-1}},\dots, \as{s}_1^{e_1}$ is an acceptable sequence of
weighted edges in $\Gamma$,
\item $\as{s}_1^{e_1},\as{s}_2^{e_2},\dots, \as{s}_i^{e_i}$ is an acceptable sequence of
weighted edges in $\Gamma$, for $i=2,\dots,n$, and,
\item for $i=1,\dots,n$, if $e_i^*=+$ when $e_i=-$, and $e_i^*=-$ when
$e_i=+$, then $\as{s}_1^{e_1^*},\as{s}_2^{e_2^*},\dots, \as{s}_{n}^{e_n^*}$ is an acceptable
sequence of weighted edges in $\Gamma$.
\end{enumerate}
\end{lemma}

Suppose that $\sigma=\as{s}_1^{e_1},\as{s}_2^{e_2},\cdots,
\as{s}_n^{e_n}$ is an acceptable sequence of weighted edges in $\Gamma$.
For $i=1,\dots,n$, let $\as{S}_i$ be the simple $\Ag$-module associated
to the edge $\as{s}_i$.  We define $\st(\sigma)$ inductively such that
the top of $\st(\sigma)$ is $\oplus_{\{i\, \mid\, e_i=+\}}\as{S}_i$ and the
socle of $\st(\sigma)$ is $\oplus_{\{i\, \mid\, e_i=-\}}\as{S}_i$.  We say
$\st(\sigma)$ satisfies the \emph{top and socle condition}.

\begin{defn}
Let $\sigma=\as{s}_1^{e_1},\as{s}_2^{e_2},\dots, \as{s}_n^{e_n}$ be an acceptable
sequence of weighted edges in $\Gamma$, and let $n\geqslant 2$.
For $n=2$, $\st(\as{s}_1^{e_1},\as{s}_2^{e_2})$ was defined above, and
clearly satisfies the top and socle condition.

Assume $n\geqslant 3$ and suppose we have defined $\st(\as{s}_1^{e_1},\as{s}_2^{e_2},\dots,
\as{s}_{n-1}^{e_{n-1}})$ satisfying the top and socle condition. There
are two cases: $e_{n-1}=-$ and $e_{n-1}=+$.
\begin{enumerate}[(i)]
\item Suppose $e_{n-1}=-$, so that $e_n=+$.
We set
\[
\st(\sigma)=\coker(\mu\colon \as{S}_{n-1}\to \st(\as{s}_1^{e_1},\dots,
\as{s}_{n-2}^{e_{n-2}}, \as{s}_{n-1}^{-})\oplus \st(\as{s}_{n-1}^{-},\as{s}_n^{+})),
\]
where $\mu$ is induced from the map on socles given by
$$\begin{cases}
\as{S}_{n-1}\to
(\as{S}_2\oplus \as{S}_4\oplus\cdots \oplus \as{S}_{n-1})\oplus \as{S}_{n-1} \mbox{ with $x\mapsto
((0,\dots,0,x),x)$} & \text{if $n$ is odd}\\
\as{S}_{n-1}\to (\as{S}_1\oplus
\as{S}_3\oplus\cdots \oplus \as{S}_{n-1})\oplus \as{S}_{n-1} \mbox{ with $x\mapsto
((0,\dots,0,x),x)$} & \text{if $n$ is even.}
\end{cases}$$

\item Now suppose that $e_{n-1}=+$.  Then $e_n=-$.   We define $\st(\sigma)$
to be the kernel of the composition
\[ \st(\as{s}_1^{e_1},\dots,\as{s}_{n-2}^{e_{n-2}},
\as{s}_{n-1}^{+})\oplus \st(\as{s}_{n-1}^{+},\as{s}_n^-)\to
(\oplus_{\{i\,\mid\,e_i=+, 1\leqslant i\leqslant n-1\}}\as{S}_i)
\oplus \as{S}_{n-1}\stackrel{\nu}{\to}\as{S}_{n-1},
\]
where the first map is given by canonical surjection onto the tops
of $\st(\as{s}_1^{e_1},\dots, \as{s}_{n-1}^{e_{n-1}})$ and
$\st(\as{s}_{n-1}^{e_{n-1}},\as{s}_n^{e_n})$ and $\nu$ is given by
$$\begin{cases}
((y_2,y_4,\dots,y_{n-1}),y_{n-1}^{\prime})\mapsto
y_{n-1}-y_{n-1}^{\prime} & \text{if $n$ is odd}\\
((y_1,y_3,\dots,y_{n-1}),y_{n-1}^{\prime})\mapsto
y_{n-1}-y_{n-1}^{\prime} &\text{if $n$ is even,}
\end{cases}$$
where $y_i\in \as{S}_i$ and
$y_{n-1}^{\prime}\in \as{S}_{n-1}$.
\end{enumerate}
\end{defn}

The reader may check that the top and socle condition is satisfied
by $\st(\sigma)$ in all cases.  The next proposition gives an
alternative definition for $\st(\sigma)$.

\begin{prop}\label{prop:st-genl}
Let $(\Gamma,\mathfrak o,\mm)$ with $\mm\equiv 1$ be a Brauer graph
with no loops or multiple edges and let $\cA_{\Gamma}$ denote the
associated Brauer graph algebra.  Suppose that $\sigma=
\as{s}_1^{e_1},\as{s}_2^{e_2},\dots, \as{s}_{n}^{e_{n}}$ is an acceptable sequence of
weighted edges in $\Gamma$ with $n\geqslant 3$ and let $\as{S}_i$ be the simple
$\Ag$-module associated to $\as{s}_i$. Let $2\leqslant k\leqslant n-1$.
\begin{enumerate}[\bf (1)]
\item If $e_{k}=-$, then we set
\[
X=\coker(\mu\colon \as{S}_{k}\to \st(\as{s}_1^{e_1},\dots,
\as{s}_{k}^{e_{k}})\oplus \st(\as{s}_{k}^{e_{k}},\dots,\as{s}_n^{e_n})),
\]
where $\mu$ is induced from the map on socles given by
$$\begin{cases}
\as{S}_{k}\to (\as{S}_2\oplus \as{S}_4\oplus\cdots \oplus \as{S}_{k})\oplus (\as{S}_{k}\oplus
\as{S}_{k+2}\oplus \cdots) \\
\quad\quad \mbox{ with } x\mapsto ((0,\dots,0,x),(x,0,
\dots,0)) & \text{if $k$ is even}\\
\as{S}_{k}\to (\as{S}_1\oplus \as{S}_3\oplus\cdots \oplus \as{S}_{k})\oplus (\as{S}_{k}\oplus \as{S}_{k+2}\oplus
\cdots) \\
\quad\quad \mbox{ with } x\mapsto ((0,\dots,0,x),(x,0,\dots,0)) & \text{if $k$ is
odd.}
\end{cases}$$
\item If $e_{k}=+$, then we set $X$ to be the kernel
of the composition of canonical surjections onto tops
\[
\st(\as{s}_1^{e_1},\dots,
\as{s}_{k}^{e_{k}})\oplus \st(\as{s}_{k}^{e_k},\dots,\as{s}_n^{e_n})\to
(\oplus_{\{e_i=+, 1\leqslant i\leqslant k\}}\as{S}_i)\oplus (\oplus_{\{e_i=+, k\leqslant i\leqslant
n\}}\as{S}_i)
\]
with
\[
\nu\colon(\oplus_{\{e_i=+, 1\leqslant i\leqslant k\}}\as{S}_i)\oplus (\oplus_{\{
e_i=+, k\leqslant i\leqslant n\}}\as{S}_i) {\to}\as{S}_{k},
\]
where $\nu$ is given by
$$\begin{cases}
((y_2,y_4,\dots,y_{k}),(y_{k}^{\prime},y_{k+2},\dots))\mapsto
y_{k}-y_{k}^{\prime} &\text{if $k$ is even}\\
((y_1,y_3,\dots,y_{k}),(y_{k}^{\prime},y_{k+2},\dots))\mapsto
y_{k}-y_{k}^{\prime} & \text{if $k$ is odd,}
\end{cases}$$
with $y_i\in \as{S}_i$ and $y_{k}^{\prime}\in \as{S}_{k}$.
\end{enumerate}
Then $\st(\sigma)\cong X$.
\end{prop}

\begin{proof}
We proceed by induction on $n$.  For $n=3$, $k$ must be $2=n-1$ and
the result follows from the definition of $\st(\sigma)$.  Now assume
the result is true for all $m$, $3\leqslant m\leqslant n-1$.  Fix $2\leqslant k\leqslant
n-1$.  If $k=n-1$, then the result again follows from the definition
of $\st(\sigma)$. Assume $k<n-1$.  There are many cases to consider:
$e_k$ equal to $+$ or $-$, $k$ even or odd, and $n$ even or odd.
The cases all have similar proofs. We prove one case and leave the
others to the reader.

The case we prove is for $n$ even, $k$ even,
and $e_k=+$.  Note that we then have, for $i$ odd, $e_i= -$ and, for $i$ even,
$e_i=+$.
For ease of notation, we set  $Z_1=\st(\as{s}_1^-,\dots,\as{s}_{n-1}^-)$,
$Z_2=\st(\as{s}_{n-1}^-,\as{s}_{n}^+)$, $U_1=\st(\as{s}_1^-,\dots,\as{s}_{k}^+)$,
$U_2=\st(\as{s}_k^+,\dots,\as{s}_{n}^+)$, and $V_1 = \st(\as{s}_k^+, \dots ,\as{s}_{n-1}^-)$.

 From the definition of $\st(\sigma)$, we have a short exact
 sequence of $\Ag$-modules
 \[
 0\to \as{S}_{n-1}\stackrel{\mu}{\to} Z_1\oplus Z_2\to \st(\sigma)\to 0.
 \]
By induction, we have a short exact sequence
\[
0\to Z_1\stackrel{g}{\to} U_1\oplus V_1\stackrel{h}{\to} \as{S}_k \to 0.
\]
From this short exact sequence we obtain
\[
\everymath{\scriptstyle}
0\to Z_1\oplus Z_2\stackrel{\left(\begin{array}{cc} g&0\\
0& \Id_{Z_2}\end{array}\right)}{\longrightarrow} (U_1\oplus V_1)
\oplus Z_2\to \as{S}_k\oplus 0 \to 0.
\]
Also by induction, we have a short exact sequence
\[
\everymath{\scriptstyle} 0\to \as{S}_{n-1}\to V_1\oplus
Z_2\stackrel{\phi}\to U_2\to 0.
\]
From this short exact sequence we obtain
\[
\everymath{\scriptstyle}
0\to 0\oplus \as{S}_{n-1}\to U_1\oplus (V_1\oplus Z_2)\stackrel{\left(\begin{array}{cc} \Id_{U_1}&0\\
0& \phi\end{array}\right)}{\longrightarrow} U_1\oplus U_2\to 0.
\]

Using the above sequences and that $U_1\oplus V_1\oplus Z_2 =
(U_1\oplus V_1)\oplus Z_2 = U_1\oplus (V_1\oplus Z_2)$,  we obtain
an exact commutative diagram:
\[
\everymath{\scriptstyle}
\xymatrix{ & & 0\ar[d]&0\ar[d]\\
0\ar[r]& \as{S}_{n-1}\ar[r]^{\mu}\ar[d]^{\cong}&Z_1\oplus
Z_2\ar[r]\ar[d]_{\left(\begin{array}{cc} g&0\\
0& \Id_{Z_2}\end{array}\right)}&\st(\sigma)\ar[r]\ar[d]&0\\
0\ar[r]&0\oplus \as{S}_{n-1}\ar[r]& U_1\oplus V_1\oplus Z_2
\ar[r]^{\left(\begin{array}{cc} \Id_{U_1}&0\\
0& \phi\end{array}\right)}\ar[d]_{\left(\begin{array}{c} h\\
0\end{array}\right)}& U_1\oplus U_2\ar[r]
\ar[d]& 0 \\
&&\as{S}_k\oplus 0\ar[r]^{\cong}\ar[d]&\as{S}_k\ar[d]\\
&&0&0 }
\]
The reader may check that the exact sequence that appears as the
last column in the diagram above proves that $X\cong \st(\sigma)$ in
this case.
\end{proof}

\begin{cor}\label{cor:string-rev}
Let $(\Gamma,\mathfrak o,\mm)$ with $\mm\equiv 1$ be a Brauer graph with no
loops or multiple edges and let $\cA_{\Gamma}$ denote the associated
Brauer graph algebra.  Suppose that $\sigma=\as{s}_1^{e_1},\as{s}_2^{e_2},\dots,
\as{s}_{n}^{e_{n}}$ is an acceptable sequence of weighted edges in
$\Gamma$ with $n\geqslant 2$.  Then $\tau
=\as{s}_n^{e_n},\as{s}_{n-1}^{e_{n-1}},\dots, \as{s}_{1}^{e_{1}}$ is also an
acceptable sequence of weighted edges in $\Gamma$ and
\[\st(\sigma)\cong \st(\tau).\]
\end{cor}

\begin{proof} By Lemma \ref{accept-seq}{\bf (1)}, $\tau$ is an acceptable sequence of weighted edges in
$\Gamma$.  For $n=2$, the result follows from the definitions. If
$n\geqslant 3$, taking $k=2$ in Proposition~\ref{prop:st-genl}, and induction
yields the result.
\end{proof}

By Corollary \ref{cor:string-rev}, we see that there are 3 types of
string modules over $\Ag$.  Suppose $\sigma=\as{s}_1^{e_1},\dots,\as{s}_n^{e_n}$ is
an acceptable sequence of weighted edges in $\Gamma$. We say that $\st(\sigma)$ is a
\emph{positive string module} if $e_1=+=e_n$, a \emph{negative
string module} if $e_1=-=e_n$, and a \emph{mixed string module} if
$e_1\ne e_n$.  We will see that the positive string modules play an
important role in the cohomology theory of $\Ag$.

\section{Syzygies and resolutions in a Brauer graph
algebra with no truncated edges}\label{sec:syzygy}

In this section we assume that
$(\Gamma,\mathfrak o,\mm)$ with $\mm\equiv 1$ is a Brauer graph with no
loops or multiple edges, and let $\cA_{\Gamma}$ denote the associated
Brauer graph algebra. The main result here is Theorem~\ref{thm:min-resol}, where we give a minimal projective resolution of a simple $\cA_{\Gamma}$-module, in the case where $\Gamma$ has no truncated edges.

We begin by fixing a nontruncated edge $s$ in $\Gamma$ with endpoints $\alpha$ and $\beta$.
Since $\mm\equiv 1$, both $\alpha$ and $\beta$ have valency at least 2.
Let edge $s'$ be in the successor sequence of $s$ at vertex
$\alpha$. We say edge $s''$ \emph{follows $(s',s)$} if $s''$ is the
successor of $s$ at vertex $\beta$. The following
diagram illustrates this definition:
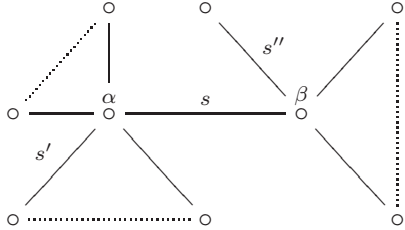
\[
\xymatrix{
&\circ\ar@{-}[d]&\circ\ar@{-}[dr]^{s''}&&\circ\ar@{-}[dl]\\
\circ\ar@{.}[ur]\ar@{-}[r]&\stackrel{\alpha}{\circ}\ar@{-}[rr]^s&&\stackrel{\beta}{\circ}\\
\circ\ar@{-}[ru]^{s'}\ar@{.}[rr]&&\circ\ar@{-}[ul]&&\circ\ar@{.}[uu]\ar@{-}[ul]
}
\]

As in Section~\ref{sec:string}, if $e=+$, we let $e^*=-$, and if $e=-$, we let $e^*=+$.
We are now in a position to describe the first syzygy  of a positive
string module.  For this, we introduce the following
notational conventions.  The string module $\st({s}^+,{t}^-,{u}^+)$  will
be schematically represented by
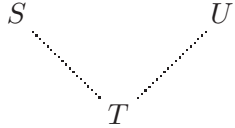
\[
\xymatrix{{S}\ar@{.}[dr]&&{U}\ar@{.}[dl]\\
&{T} }\]
For a nontruncated edge ${s}$, the indecomposable projective $\Ag$-module with top ${S}$ will
be schematically represented by
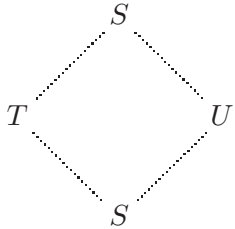
\[\xymatrix{
&{S}\ar@{.}[dl]\ar@{.}[dr]\\
{T}\ar@{.}[dr]&&{U}\ar@{.}[dl]\\
&{S} }\] where the edge ${t}$ in $\Gamma$ is in the successor sequence
for ${s}$ at one endpoint of ${s}$, and the edge ${u}$ is in the successor
sequence for ${s}$ at the other endpoint of ${s}$.  If a solid line
\[\xymatrix{
{S}\ar@{-}[d]\\{T} }\]
appears, then that signifies that not only is ${t}$ in the successor sequence
of ${s}$ at some vertex of $\Gamma$, but ${t}$ is the
successor of ${s}$ at that vertex.

\begin{prop}\label{prop:first-sy}
Let $(\Gamma,\mathfrak o,\mm)$ with $\mm\equiv 1$ be a
Brauer graph with no loops or multiple edges, and let
$\cA_{\Gamma}$ denote the associated Brauer graph algebra.
Suppose
that $M=\st(\sigma)$ is a positive string module, where
$\sigma=\as{s}_1^{e_1},\dots, \as{s}_{n}^{e_{n}}$ is an acceptable sequence
of weighted edges in $\Gamma$, and $n\geqslant 3$.
Suppose also, for each $i$ with $e_i = +$, that $\as{s}_i$ is a nontruncated edge.
Then the first syzygy
of $M$ is isomorphic to the positive string module $\st(\tau)$,
where
\[\tau = \as{s}_0^{e_1},\as{s}_1^{e_1^*},\dots,
\as{s}_{n}^{e_{n}^*},\as{s}_{n+1}^{e_n},\]
and where $\as{s}_0$ follows $(\as{s}_2,\as{s}_1)$ and $\as{s}_{n+1}$ follows $(\as{s}_{n-1},\as{s}_n)$.
\end{prop}

\begin{proof}
Since $\st(\sigma)$ is a positive string module, $n$ is an odd integer, say $n=2m+1$. By
assumption, $m\geqslant 1$.  The projective cover of $M$ is $P_{\as{S}_1}\oplus
P_{\as{S}_3}\oplus \cdots \oplus P_{\as{S}_{2m+1}}$ and the socle of $M$ is given by
$\as{S}_2\oplus \as{S}_4\oplus\cdots \oplus \as{S}_{2m}$.  Thus $M$ structurally looks
like:
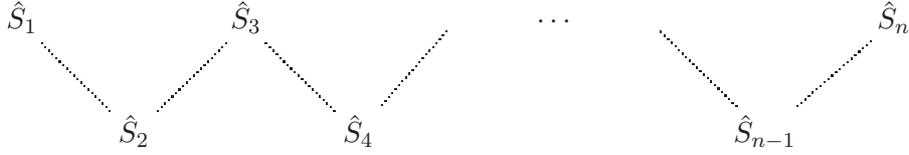
\[
\xymatrix{  \as{S}_1\ar@{.}[dr]&&\as{S}_3\ar@{.}[ld]\ar@{.}[rd]&&&\cdots&\ar@{.}[dr]&&\as{S}_{n}\ar@{.}[dl]\\
&\as{S}_2&&\as{S}_4\ar@{.}[ur]&&&&\as{S}_{n-1}\ar@{.}[ur]
 }\]
Since $\as{s}_1, \as{s}_3, \dots , \as{s}_{2m+1}$ are all nontruncated edges, the corresponding indecomposable projectives $P_{\as{S}_1}, P_{\as{S}_3}, \dots , P_{\as{S}_{2m+1}}$ are biserial. Thus, from the definition of `follows',
$P_{\as{S}_1}\oplus P_{\as{S}_3}\oplus \cdots \oplus P_{\as{S}_{2m+1}}$ looks like:

{\tiny
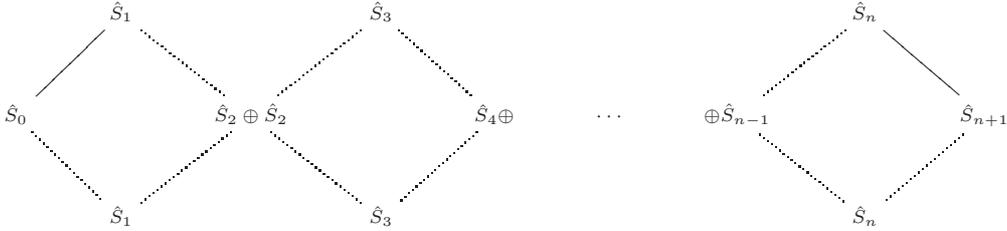
\[
\xymatrix{ &\as{S}_1\ar@{-}[dl]\ar@{.}[dr]&&\as{S}_3\ar@{.}[dl]\ar@{.}[dr]
&&&&\as{S}_n\ar@{.}[dl]\ar@{-}[dr]
\\
\as{S}_0\ar@{.}[dr] &&\as{S}_2\ar@{.}[dl]\oplus
 \as{S}_2\ar@{.}[dr]&&\as{S}_4\ar@{.}[dl]\oplus&\cdots&\oplus
\as{S}_{n-1}\ar@{.}[dr]&& \as{S}_{n+1}\ar@{.}[dl]
\\
&\as{S}_1&&\as{S}_3&&&&\as{S}_n
 }
\]}
From these diagrams, the reader can easily provide the remaining
details of the proof.
\end{proof}

We assume for the rest of this section that $\Gamma$ contains no truncated
edges.

\bigskip

To describe projective resolutions of simple $\Ag$-modules, we will
need further notation. For ${s}$ an edge in $\Gamma$, we represent the
simple $\Ag$-module ${S}$ by $v_{{s}}(\Ag/\br)$, where $v_{{s}}$ is the vertex
in $\cQ_{\Gamma}$ associated to the edge ${s}$. Here we are viewing
$v_{{s}}$ as the idempotent in $\Ag$ corresponding to the edge ${s}$ in
$\Gamma$. We also set the projective $\Ag$-module $P_{{S}}$ to be
$v_{{s}}(\Ag)$.

Let $\as{s}_0$ be an edge in $\Gamma$. We now present a minimal
projective $\Ag$-resolution of the simple module $v_{\as{s}_0}(\Ag/\br)$,
\[(Q^{\bullet},f^{\bullet}):\quad \cdots \to Q^2\stackrel{f^2}{\to}
Q^1\stackrel{f^1}{\to}Q^0\stackrel{f^0}{\to} v_{\as{s}_0}(\Ag/\br)\to 0
.\] We see that $Q^0= v_{\as{s}_0}(\Ag)$ with $f^0$ being the canonical
surjection and the first syzygy is $\st(\as{s}_{-1}^+,\as{s}_0^-,\as{s}_1^{+})$,
where $\as{s}_{-1}$ and $\as{s}_{1}$ are the successors of $\as{s}_0$ at its
endpoints. Applying Proposition \ref{prop:first-sy} repeatedly, we see
that, if $n$ is odd, then the $n$-th syzygy of $ v_{\as{s}_0}(\Ag/\br)$ is
\[\Omega^n_{\Ag}(v_{\as{s}_0}(\Ag/\br)
)=\st(\as{s}_{-n}^+,\as{s}_{-n+1}^-,\as{s}_{-n+2}^+,\dots,\as{s}_{-1}^+,\as{s}_0^-,\as{s}_1^+,\dots,\as{s}_{n-1}^-,\as{s}_n^+),
\]
and, if $n$ is even, then the $n$-th syzygy of $ v_{\as{s}_0}(\Ag/\br)$ is
\[\Omega^n_{\Ag}(v_{\as{s}_0}(\Ag/\br)
)=\st(\as{s}_{-n}^+,\as{s}_{-n+1}^-,\as{s}_{-n+2}^+,\dots,\as{s}_{-1}^-,\as{s}_0^+,\as{s}_1^-,\dots,\as{s}_{n-1}^-,\as{s}_n^+),
\]
where, for $i=2,\dots,n$, $\as{s}_{-i}$ follows $(\as{s}_{-i+2},\as{s}_{-i+1})$,
and $\as{s}_i$ follows $(\as{s}_{i-2},\as{s}_{i-1})$. From this we obtain the next result.

\begin{prop}\label{prop:projQ}
Keeping the above notation, let $Q^n$ be the $n$-th projective in a minimal projective $\Ag$-resolution
of the simple module $v_{\as{s}_0}(\Ag/\br)$, and let $n > 0$.
If $n$ is odd, \[ Q^n=
P_{\as{S}_{-n}}\oplus P_{\as{S}_{-n+2}}\oplus\cdots \oplus P_{\as{S}_{-1}}\oplus
P_{\as{S}_{1}} \oplus\cdots \oplus P_{\as{S}_{n-2}}\oplus P_{\as{S}_{n}},
\]
and, if $n$ is even, \[ Q^n= P_{\as{S}_{-n}}\oplus
P_{\as{S}_{-n+2}}\oplus\cdots \oplus P_{\as{S}_{-2}}\oplus P_{\as{S}_0}\oplus
P_{\as{S}_{2}} \oplus\cdots \oplus P_{\as{S}_{n-2}}\oplus P_{\as{S}_{n}},
\]
where, $\as{s}_{-1}$ and $\as{s}_{1}$ are the successors of $\as{s}_0$ at its
endpoints, and, for $i=2,\dots,n$, $\as{s}_{-i}$ follows $(\as{s}_{-i+2},\as{s}_{-i+1})$,
and $\as{s}_i$ follows $(\as{s}_{i-2},\as{s}_{i-1})$.
\end{prop}

It remains to describe the maps $f^n$ in the projective resolution.
We recall from Section~\ref{sec:coverings} that, if edge $t$ is the successor of edge
$s$ in $\Gamma$ at vertex $\alpha$, then we denote the corresponding
arrow in $\cQ_{\Gamma}$ from vertex $v_s$ to vertex $v_t$
by $a(s,t)$, since there are no loops or multiple edges in $\Gamma$. \label{convention on arrows}
Suppose that $s=s_0,s_1,s_2,\dots,s_{n-1}$ is the successor sequence for $s$ at the
vertex $\alpha$ in $\Gamma$, and set $s_n = s_0$.  If $1\leqslant k\leqslant n-1$, we denote the
path $a(s_0,s_1)a(s_1,s_2)\cdots a(s_{k-1},s_k)$, from
$v_{s_0}$ to $v_{s_k}$ in $\cQ_{\Gamma}$, by $p(s_0,s_k)$. Note that our
assumptions on $\Gamma$ show that $p(s_0,s_k)$ is well-defined.

\begin{lemma}\label{lem:prod} Suppose that $s$ is an edge in $\Gamma$
with endpoints $\alpha$ and $\beta$ and that
$s=s_0,s_1,\dots,s_{n-1}$ is the successor sequence for $s$ at
$\alpha$. Let $s_n = s_0$ since $s_0$ is the successor of $s_{n-1}$.
Assume that $t$ is in the successor sequence of $s$ at $\beta$.
\begin{enumerate}[\bf (1)]
\item If $0\leqslant i<j<k\leqslant n$, then
$p(s_i,s_j)p(s_j,s_k)=p(s_i,s_k)\ne 0$.
\item If $1\leqslant i\leqslant n$, then $p(t,s_0)p(s_0,s_i)=0$.
\item If $0\leqslant i\leqslant n-1$, then $p(s_i,s_n)p(s_0,t)=0$.
\end{enumerate}
\end{lemma}

\begin{proof} To prove {\bf (1)}, we note that $p(s_i,s_j)=
a(s_{i},s_{i+1}) \cdots a(s_{j-1},s_{j})$ and $p(s_j,s_k)=
a(s_{j},s_{j+1}) \cdots a(s_{k-1},s_{k})$. Hence
$p(s_i,s_j)p(s_j,s_k)=p(s_i,s_k)$. That $p(s_i,s_k)\ne 0$ follows
from the relations defining $I_{\Gamma}$ and the fact that $p(s_i,s_k)$ is a
factor of $C_{s,\alpha}$.

The other parts follow from the relations defining $I_{\Gamma}$, and the fact
that $p(s_0,s_i)$ and $p(s_i,s_n)$ are associated to the successor sequence for $s$ at
vertex $\alpha$, whereas the paths $p(t,s_0)$ and $p(s_0,t)$ are associated
to the successor sequence for $s$ at vertex $\beta$.
\end{proof}

We are now in a position to define the maps $f^n\colon Q^n\to
Q^{n-1}$, for $n\geqslant 0$. The map $f^0\colon v_{\as{s}_0}(\Ag)\to
v_{\as{s}_0}(\Ag/\br)$ is the canonical surjection.  Recall that, for
each edge $s\in\Gamma$, we are setting $P_{S}=v_{s}(\Ag)$.  Using
our description of the projective module $Q^n$ in Proposition~\ref{prop:projQ}, for
$n\geqslant 1$, we write $Q^n$ as an $(n+1) \times 1$ column vector. Then
$f^n$ will be given as an $n\times (n+1)$ matrix with the
$(i,j)$-th entry in $v_{\as{s}_{-n+2i-1}}(\Ag) v_{\as{s}_{-n+2j-2}}$,
representing a map from $P_{\as{S}_{-n+2j-2}}= v_{\as{s}_{-n+2j-2}}(\Ag)$ to
$P_{\as{S}_{-n+2i-1}}=v_{\as{s}_{-n+2i-1}}(\Ag)$.

The map $f^1\colon Q^1\to
Q^0$ is given by the $1\times 2$ matrix
\[\left(\begin{array}{cc} p(\as{s}_0,\as{s}_{-1})& p(\as{s}_0,\as{s}_{1})
\end{array}\right),\]
where $\as{s}_{-1}$ and $\as{s}_1$ are the successors of $\as{s}_0$ at its endpoints,
so that $p(\as{s}_0,\as{s}_{-1})=a(\as{s}_0,\as{s}_{-1})$ and $p(\as{s}_0,\as{s}_{1})= a(\as{s}_0,\as{s}_1)$.

For $n\geqslant 2$, $f^n$ is given by the matrix

{\tiny
\[\left(\begin{array}{cccccc}
(-1)^{n-1}p(\as{s}_{-n+1},\as{s}_{-n}) & p(\as{s}_{-n+1},\as{s}_{-n+2})&0&\cdots&0&0\\
0&(-1)^{n-1}p(\as{s}_{-n+3},\as{s}_{-n+2}) & p(\as{s}_{-n+3},\as{s}_{-n+4})&&0&0\\
0&0&(-1)^{n-1}p(\as{s}_{-n+5},\as{s}_{-n+4}) & \cdots&0&0\\
0&0&0&&0&0\\
\vdots &\vdots&\vdots&&\vdots&\vdots\\
 &&&&p(\as{s}_{n-3},\as{s}_{n-2})&0\\
 0&0&0&\cdots&(-1)^{n-1}p(\as{s}_{n-1},\as{s}_{n-2})&p(\as{s}_{n-1},\as{s}_n)
\end{array}\right).
\]}

We now come to the main result of this section, which shows that we have indeed described
a minimal projective $\Ag$-resolution of the simple module $v_{\as{s}_0}(\Ag/\br)$.

\begin{thm}\label{thm:min-resol}
Let  $(\Gamma,\mathfrak o,\mm)$ with $\mm\equiv 1$ be a Brauer graph with no
loops or multiple edges and no truncated edges, and let
$\cA_{\Gamma}$ denote the associated Brauer graph algebra. Let $\as{s}_0$
be an edge in $\Gamma$ and
\[(Q^{\bullet},f^{\bullet}):\quad \cdots \to Q^2\stackrel{f^2}{\to}
Q^1\stackrel{f^1}{\to}Q^0\stackrel{f^0}{\to} v_{\as{s}_0}(\Ag/\br)\to 0
\] be as given above.  Then $(Q^{\bullet},f^{\bullet})$ is a
minimal projective  $\Ag$-resolution of $v_{\as{s}_0}(\Ag/\br)$.
\end{thm}

\begin{proof}
We begin by showing that $f^{n-1}\circ f^n=0$, for
$n\geqslant 1$. For $n=1$ this is clear since, from the definitions of
$f^0$ and $f^1$, we see that
$\Im(f^1)=v_{\as{s}_{0}}\br=\st(\as{s}_{-1}^+,\as{s}_0^-,\as{s}_1^+)=\Ker(f^0)$.
That $f^1\circ f^2=0$ can be proved directly from the matrices. So
assume $n\geqslant 3$.

Let $A$ be the matrix representing $f^n$

{\tiny
\[\left(\begin{array}{cccccc}
(-1)^{n-1}p(\as{s}_{-n+1},\as{s}_{-n}) & p(\as{s}_{-n+1},\as{s}_{-n+2})&0&\cdots&0&0\\
0&(-1)^{n-1}p(\as{s}_{-n+3},\as{s}_{-n+2}) & p(\as{s}_{-n+3},\as{s}_{-n+4})&&0&0\\
0&0&(-1)^{n-1}p(\as{s}_{-n+5},\as{s}_{-n+4}) & \cdots&0&0\\
0&0&0&&0&0\\
\vdots &\vdots&\vdots&&\vdots&\vdots\\
 &&&&p(\as{s}_{n-3},\as{s}_{n-2})&0\\
 0&0&0&\cdots&(-1)^{n-1}p(\as{s}_{n-1},\as{s}_{n-2})&p(\as{s}_{n-1},\as{s}_n)
\end{array}\right)
\]}
and $B$ be the matrix representing $f^{n-1}${\tiny
\[\left(\begin{array}{cccccc}
(-1)^{n-2}p(\as{s}_{-n+2},\as{s}_{-n+1}) & p(\as{s}_{-n+2},\as{s}_{-n+3})&0&\cdots&0&0\\
0&(-1)^{n-2}p(\as{s}_{-n+4},\as{s}_{-n+3}) & p(\as{s}_{-n+4},\as{s}_{-n+5})&&0&0\\
0&0&(-1)^{n-2}p(\as{s}_{-n+6},\as{s}_{-n+5}) & \cdots&0&0\\
0&0&0&&0&0\\
\vdots &\vdots&\vdots&&\vdots&\vdots\\
 &&&&p(\as{s}_{n-4},\as{s}_{n-3})&0\\
 0&0&0&\cdots&(-1)^{n-2}p(\as{s}_{n-2},\as{s}_{n-3})&p(\as{s}_{n-2},\as{s}_{n-1})
\end{array}\right).
\]}
We show $BA$ is the zero matrix.  The $(1,1)$-entry of $BA$
is ${-p(\as{s}_{-n+2},\as{s}_{-n+1})p(\as{s}_{-n+1},\as{s}_{-n)}}$. But $\as{s}_{-n+2}$
and $\as{s}_{-n}$ are in the successor sequences for $\as{s}_{-n+1}$ at different
vertices.  Hence $p(\as{s}_{-n+2},\as{s}_{-n+1})p(\as{s}_{-n+1},\as{s}_{-n)}=0$ by
Lemma~\ref{lem:prod}{\bf (3)}. The remaining entries of the first column in $BA$
are all $0$.

The $(1,2)$-entry of $BA$ is
\[(-1)^{n-2}p(\as{s}_{-n+2},\as{s}_{-n+1}) p(\as{s}_{-n+1},\as{s}_{-n+2}) + (-1)^{n-1}p(\as{s}_{-n+2},\as{s}_{-n+3})
 p(\as{s}_{-n+3},\as{s}_{-n+2}).\]
Suppose the endpoints of $\as{s}_{-n+2}$ are $\alpha$ and $\beta$ in $\Gamma$.
If $\as{s}_{-n+1}$ is in the successor sequence of $\as{s}_{-n+2}$ at the vertex $\alpha$,
then $p(\as{s}_{-n+2},\as{s}_{-n+1})p(\as{s}_{-n+1},\as{s}_{-n+2}) = C_{\as{s}_{-n+2},\alpha}$. We must also have that
$\as{s}_{-n+3}$ is in the successor sequence of $\as{s}_{-n+2}$ at the vertex $\beta$,
and $p(\as{s}_{-n+2},\as{s}_{-n+3})p(\as{s}_{-n+3},\as{s}_{-n+2}) = C_{\as{s}_{-n+2},\beta}$.
Hence we see that the $(1,2)$-entry of $BA$ is
$(-1)^{n-2}(C_{\as{s}_{-n+2},\alpha} - C_{\as{s}_{-n+2},\beta})$ which is a scalar multiple of a
relation of type one in $I_{\Gamma}$ and hence $0$.

\sloppy The $(2,2)$-entry of $BA$ is $(-1)^{n-2}p(\as{s}_{-n+4},\as{s}_{-n+3})(-1)^{n-1}p(\as{s}_{-n+3},\as{s}_{-n+2})$.
But $\as{s}_{-n+4}$ and
$\as{s}_{-n+2}$ are in the successor sequences for $\as{s}_{-n+3}$ at different
vertices.  Hence $p(\as{s}_{-n+4},\as{s}_{-n+3})p(\as{s}_{-n+3},\as{s}_{-n+2})=0$ by Lemma \ref{lem:prod}. The remaining entries of
the second column in $BA$ are all $0$.

This alternating pattern continues for the remaining columns and we have
shown $BA$ is the zero matrix. Thus $\Im(f^n)\subseteq
\Ker(f^{n-1})$.

To show equality, we note that the top of $Q^n$ maps into
$\Ker(f^{n-1})$.  Inductively, we may assume that $\Ker(f^{n-1})$ is
isomorphic to $\Omega^{n}_{\Ag}(v_{\as{s}_0}(\Ag/\br))$.  The uniqueness of the simple
composition factors of indecomposable projective $\Ag$-modules given
in Lemma~\ref{up-low-uniser}{\bf (3)}, together with the structure of the
syzygies given in Proposition~\ref{prop:first-sy}, show that the top of
$Q^n$ maps isomorphically to the top of $\Ker(f^{n-1})$. Hence $\Im(f^n) =
\Ker(f^{n-1})$.

Since, for $n\geqslant 1$, the image of $f^n$ is contained in $Q^{n-1}\br$,
the resolution is minimal and the proof is complete.
\end{proof}

\section{The Ext algebra of a Brauer graph algebra with no truncated edges}\label{sec:coh-ring}

In this section, we assume that $(\Gamma,\mfo,\mm)$
is a Brauer graph with no truncated edges. We prove one of the main results of this paper, showing that the Ext algebra of the associated Brauer graph algebra $\cA_\Gamma$ is finitely generated in degrees 0, 1 and 2.

Let $G$ be a finite abelian group and let
$W\colon \cZ_{\Gamma}\to G$ be a Brauer weighting such that the
Brauer covering graph $(\Gamma_W,\mathfrak o_W,\mm_W)$ has $\mm_W\equiv 1$ and no loops or
multiple edges (see Section~\ref{sec:coverings}). Suppose that $s$ is an edge in $\Gamma$ incident with vertex $\alpha$ in
$\Gamma$, and $s_g$ is an edge in $\Gamma_W$ incident with vertex $\alpha_g$ in $\Gamma_W$, such that
$s_g$ lies over $s$ and $\alpha_g$ lies over
$\alpha$ for some $g\in G$. Then, by \cite[Proposition 4.4 and Definition 4.5]{GSS},
$\mm(\alpha)\val_{\Gamma}(\alpha)=
\mm_W(\alpha_g)\val_{\Gamma_W}(\alpha_g)$, where
$\val_{\Gamma}(\alpha)$ and $\val_{\Gamma_W}(\alpha_g)$ denote the
valencies of $\alpha$ and $\alpha_g$ respectively. It follows that
edge $s$ in $\Gamma$ is truncated at vertex $\alpha$ in $\Gamma$ if and only if edge
$s_g$ in $\Gamma_W$ is truncated at vertex $\alpha_g$
in $\Gamma_W$.
Thus, to study the Ext algebra of a Brauer graph algebra associated to a Brauer
graph $(\Gamma,\mfo,\mm)$ with no truncated edges, it follows from the above
discussion and Proposition~\ref{prop:ext-struct} that we may assume $(\Gamma,\mfo,\mm)$ is a Brauer graph with $\mm\equiv 1$, with no loops or multiple edges and no truncated edges.

Let $\Ag$ denote
the associated Brauer graph algebra and let
$\br$ denote the Jacobson radical of $\Ag$. The Ext algebra of $\Ag$ is $E(\Ag) = \oplus_{n\geqslant
0}\Ext^n_{\Ag}(\Ag/\br,\Ag/\br)$ with the Yoneda product.
Let $\as{s}_0$ be an edge in $\Gamma$ and $\as{S}_0=v_{\as{s}_0}(\Ag/\br)$
the associated simple $\Ag$-module, where $v_{\as{s}_0}$ is the
idempotent in $\Ag$ associated to $\as{s}_0$. Let
\[(Q^{\bullet},f^{\bullet}):\quad \cdots \to Q^2\stackrel{f^2}{\to}
Q^1\stackrel{f^1}{\to}Q^0\stackrel{f^0}{\to} v_{\as{s}_0}(\Ag/\br)\to 0
\]
be the minimal projective $\Ag$-resolution of $\as{S}_0$ given in Theorem~\ref{thm:min-resol}, with \[Q^n=\begin{cases}P_{\as{S}_{-n}}\oplus
P_{\as{S}_{-n+2}}\oplus\cdots\oplus P_{\as{S}_{-2}}\oplus P_{\as{S}_0}\oplus
P_{\as{S}_2}\oplus\cdots\oplus P_{\as{S}_n},&\text{for $n$ even}\\
P_{\as{S}_{-n}}\oplus P_{\as{S}_{-n+2}}\oplus\cdots\oplus P_{\as{S}_{-1}}\oplus
P_{\as{S}_1}\oplus\cdots\oplus P_{\as{S}_n},&\text{for $n$ odd.}
\end{cases}
\]
Since each $P_{\as{S}}$ is an indecomposable projective $\Ag$-module,
we choose a $K$-basis $G^n_i(\as{S}_0)$, where $i\in\{-n,-n+2,\dots, n-2,n\}$, for
$\Ext^n_{\Ag}(\as{S}_0,\Ag/\br)$, where $G^n_i(\as{S}_0)$ represents the
element in $\Ext^n_{\Ag}(\as{S}_0,\Ag/\br)$ given by the composition
\[ Q^n\to P_{\as{S}_i}\to \as{S}_i\to \Ag/\br,\] where the first map is
the projection map, the second map is the canonical surjection, and
the third map is inclusion.  We call
$\{G_i^n(\as{S}_0)\}_{i\in\{-n,-n+2,\dots, n-2,n\}}$ the \emph{canonical basis
of $\Ext^n_{\Ag}(\as{S}_0,\Ag/\br)$}. Since
$\Ag/\br=\oplus_{\as{s}_0\in\Gamma_1}\as{S}_0$, so that
\[\Ext^n_{\Ag}(\Ag/\br,\Ag/\br)=\bigoplus_{\as{s}_0\in\Gamma_1}\Ext^n_{\Ag}(\as{S}_0,\Ag/\br)
,\] we abuse notation and view
\[
\cG^n=\bigcup_{\as{s}_0\in\Gamma_1}\{G_i^n(\as{S}_0)\}_{i\in\{-n,-n+2,\dots,
n-2,n\}}
\]
as a $K$-basis of  $\Ext^n_{\Ag}(\Ag/\br,\Ag/\br)$.

We now present the main result of this section.

\begin{thm}\label{thm:fg-ext}
Let $(\Gamma,\mfo,\mm)$ be a Brauer graph with no truncated edges, and let $\Ag$ denote the associated Brauer graph
algebra. Then the Ext algebra, $E(\Ag)$, is finitely generated in degrees $0, 1$ and $2$.
\end{thm}

\begin{proof}
From the above discussion, we may assume $(\Gamma,\mfo,\mm)$ is a Brauer graph with $\mm\equiv 1$, with no loops or multiple edges and no truncated edges. Fix an edge $\as{s}_0$ in $\Gamma$ with associated simple $\Ag$-module
$\as{S}_0$. We keep the previous notation. Since $\Ag/\br=\oplus_{t\in\Gamma_1}T$, we have
$\Ext^n_{\Ag}(\as{S}_0,\Ag/\br)=\oplus_{t\in\Gamma_1}\Ext^n_{\Ag}(\as{S}_0,T)$,
and hence, for $i\in\{-n,-n+2,\dots,n-2,n\}$, we may view
$G^n_i(\as{S}_0)$ as a map $G^n_i(\as{S}_0)\colon Q^n\to \as{S}_i$.

First we suppose that $n\geqslant 2$ is an even integer. Let
$i\in\{2,4,\dots,n-2,n\}$, and let
\[(R^{\bullet},g^{\bullet}):\quad \cdots \to R^2\stackrel{g^2}{\to}
R^1\stackrel{g^1}{\to}R^0\stackrel{g^0}{\to} v_{\as{s}_{i-1}}(\Ag/\br)\to
0
\]
be the minimal projective $\Ag$-resolution of $\as{S}_{i-1}$ given in
Theorem~\ref{thm:min-resol}.  We show that
\[ G^n_i(\as{S}_0)=G^{n-1}_{i-1}(\as{S}_0)*G^1_i(\as{S}_{i-1}),\]
where the right hand side is viewed in the Yoneda product
$\Ext^1_{\Ag}(\as{S}_{i-1},\as{S}_i)\times \Ext^{n-1}_{\Ag}(\as{S}_{0},\as{S}_{i-1})$.
For ease of notation and consistency, we set $\as{t}_0=\as{s}_{i-1}$ (noting that $i \neq 0$). Let
$\as{t}_1$ and $\as{t}_{-1}$ be the edges in $\Gamma$ that are the successors
of $\as{t}_0$ at the endpoints of $\as{t}_0$.  Define the sequence
$\as{t}_{-n},\as{t}_{-n+1},\dots,\as{t}_{n-1},\as{t}_n$ recursively:
$$\mbox{for $i>1$, $\as{t}_i$ follows $(\as{t}_{i-2},\as{t}_{i-1})$ and $\as{t}_{-i}$ follows
$(\as{t}_{-i+2},\as{t}_{-i+1})$}.$$
With this notation,
\[R^m=\begin{cases}P_{\as{T}_{-m}}\oplus P_{\as{T}_{-m+2}}\oplus\cdots\oplus

P_{\as{T}_{-2}}\oplus P_{\as{T}_0}\oplus
P_{\as{T}_2}\oplus\cdots\oplus P_{\as{T}_m},&\text{for $m$ even}\\
P_{\as{T}_{-m}}\oplus P_{\as{T}_{-m+2}}\oplus\cdots\oplus P_{\as{T}_{-1}}\oplus
P_{\as{T}_1}\oplus\cdots\oplus P_{\as{T}_m},&\text{for $m$ odd,}
\end{cases}
\]
and the maps $g^m\colon R^m \to R^{m-1}$ are given in a similar fashion to the maps $f^n$ in the resolution of $\as{S}_0$.

We begin by finding maps $\psi_0$ and $\psi_1$ such that the
following diagram commutes.
\[(*)\quad\quad\xymatrix{
Q^n\ar[r]^{f^n}\ar[d]^{\psi_1}&Q^{n-1}\ar[d]^{\psi_0}\ar[dr]^{G^{n-1}_{i-1}(\as{S}_0)}\\
R^1\ar[r]^{g^1}&R^0\ar[r]^{g^0}&\as{S}_{i-1} }\]

Since $\as{t}_0=\as{s}_{i-1}$ and since $\as{s}_i$ follows $(\as{s}_{i-2},\as{s}_{i-1})$, we
see that $\as{s}_i$ is the successor of $\as{s}_{i-1}$ at one of its
endpoints. Furthermore, $\as{s}_{i-2}$ is in the successor sequence of
$\as{s}_{i-1}$ at the other endpoint $\beta$ of $\as{s}_{i-1}$.  Thus, after
reordering, we may assume that $\as{t}_1=\as{s}_i$ and that both $\as{s}_{i-2}$ and
$\as{t}_{-1}$ are in the successor sequence of $\as{s}_{i-1}$ at the vertex
$\beta$, with $\as{t}_{-1}$ being the successor of $\as{s}_{i-1}$
at $\beta$.  The following diagram illustrates this.
$$
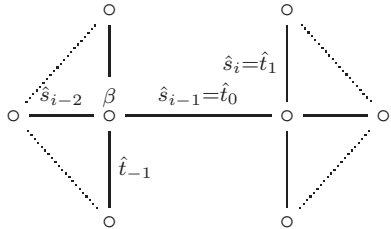
\xymatrix{
&\circ\ar@{-}[d]&&\circ\ar@{-}[d]_{\as{s}_i=\as{t}_1}\ar@{.}[dr]\\
\circ\ar@{.}[ur]\ar@{-}[r]^{\as{s}_{i-2}}&\stackrel{\beta}{\circ}\ar@{-}[rr]^{\as{s}_{i-1}=\as{t}_0}&&{\circ}\ar@{-}[r]&\circ\\
&\circ\ar@{.}[lu]\ar@{-}[u]_{\as{t}_{-1}}&&\circ\ar@{-}[u]\ar@{.}[ur] }
$$
From this we see that  $ R^0=P_{\as{T}_0}=P_{\as{S}_{i-1}}$ and
$R^1=P_{\as{T}_{-1}}\oplus P_{\as{T}_1}=P_{\as{T}_{-1}}\oplus P_{\as{S}_{i}}$. Define
$\psi_0\colon Q^{n-1}\to R^0$ by
$\psi_0(x_{-n+1},x_{-n+3},\dots,x_{n-3},x_{n-1})=x_{i-1}$ and define
$\psi_1\colon Q^n\to R^1$ by the $2\times n$ matrix
\[\left(\begin{array}{cccccccc}
0&\cdots&0&-p({\as{t}_{-1}},\as{s}_{i-2})&0&0&\cdots&0\\
0&\cdots&0&0&v_{\as{s}_i}&0&\cdots&0\end{array}\right)
\] where the first nonzero column represents the map from
$P_{\as{S}_{i-2}}$ to $P_{\as{T}_{-1}}\oplus P_{\as{S}_i}$ and the next nonzero
column represents the  map from $P_{\as{S}_{i}}$ to $P_{\as{T}_{-1}}\oplus
P_{\as{S}_i}$.  The reader may now check the commutativity of ($*$).

Next we see from the diagram
\[\xymatrix{
Q^n\ar[r]^{f^n}\ar[d]^{\psi_1}&Q^{n-1}\ar[d]^{\psi_0}\ar[dr]^{G^{n-1}_{i-1}(\as{S}_0)}\\
R^1\ar[dr]_{G^1_i(\as{S}_{i-1})}\ar[r]^{g^1}&R^0\ar[r]^{g^0}&\as{S}_{i-1}\\
&\as{S}_{i}
 }\]
that $G^1_i(\as{S}_{i-1})\circ\psi_1 =
G^{n-1}_{i-1}(\as{S}_0)*G^1_i(\as{S}_{i-1})=G^n_i(\as{S}_0)$.

Next suppose $i\in\{-n,-n+2,\dots,-2\}$. The proof that
$G^{n-1}_{i+1}(\as{S}_0)*G^1_{i}(\as{S}_{i+1})=G^n_i(\as{S}_0)$ is similar and is left
to the reader.  In fact, interchanging $\as{s}_i$ with $\as{s}_{-i}$ for
$i\in \{2,\dots,n-2,n\}$ in the above proof, gives the result.

For $n$ even, it remains to consider the case when  $i=0$. Here we show that
\[G^{n-2}_{0}(\as{S}_0)*G^2_0(\as{S}_{0})=G^n_0(\as{S}_0).\]
This is clear if $n=2$
so assume $n\geqslant 4$. We find explicit maps $\theta_0$, $\theta_1$, and $\theta_2$ such
that the following diagram commutes.
\[(**)\quad\quad\xymatrix{
Q^n\ar[r]^{f^n}\ar[d]^{\theta_2}&Q^{n-1}\ar[d]^{\theta_1}\ar[r]^{f^{n-1}}&Q^{n-2}\ar[d]^{\theta_0}
\ar[dr]^{G^{n-2}_{0}(\as{S}_0)}\\
Q^2\ar[r]^{f^2}&Q^1\ar[r]^{f^1}&Q^0\ar[r]^{f^0}&\as{S}_{0} }\] We have
that $Q^i=P_{\as{S}_{-i}}\oplus P_{\as{S}_{-i+2}}\oplus \cdots \oplus
P_{\as{S}_{i-2}}\oplus P_{\as{S}_{i}}$. Thus, for $i\geqslant 3$, and $j=i-2k$, with
$k\geqslant 0$,
\[Q^i=P_{\as{S}_{-i}}\oplus P_{\as{S}_{-i+2}}\oplus\cdots \oplus P_{\as{S}_{-j-2}} \oplus
Q^j\oplus P_{\as{S}_{j+2}}\oplus\cdots\oplus P_{\as{S}_{i}}.\] We then take $\theta_0$,
$\theta_1$, and $\theta_2$ to be the projections $Q^i\to  Q^{j}$, for the
appropriate $i$'s and $j$'s. The reader may check that $(**)$
commutes. Noting that the composition $G^{2}_0(\as{S}_0)\circ \theta_2$ is
just $G^n_0(\as{S}_0)$, we have that
$G^{n-2}_{0}(\as{S}_0)*G^2_0(\as{S}_{0})=G^n_0(\as{S}_0)$.
This completes the study of the case when $n$ is even.

Now suppose
that $n$ is odd,  $n\geqslant 3$.  We claim that, for
$i\in\{1,3,\dots,n-2,n\}$,
\[G^{n-1}_{i-1}(\as{S}_0)*G^1_i(\as{S}_{i-1})=G^n_i(\as{S}_0),\] and
for $i\in\{-1,-3,\dots,-n+2,-n\}$,
\[G^{n-1}_{i+1}(\as{S}_0)*G^1_i(\as{S}_{i+1})=G^n_i(\as{S}_0).\]
For $i \neq -1,1$, the proof is analogous to the $n$ even
case.
For $i=-1$ or  $i=+1$, keeping
$\theta_0$ and $\theta_1$ as above, we get a commutative diagram
\[\xymatrix{
Q^n\ar[r]^{f^n}\ar[d]^{\theta_1}&Q^{n-1}\ar[d]^{\theta_0}
\ar[dr]^{G^{n-1}_{0}(\as{S}_0)}\\
Q^1\ar[r]^{f^1}&Q^0\ar[r]^{f^0}&\as{S}_{0} }\]
and it is now immediate that
\[G^{n-1}_{0}(\as{S}_{0})*G^1_i(\as{S}_{0})=G^n_i(\as{S}_0).\]

Thus we have shown that every basis element in $\Ext^n_{\Ag}(\as{S}_0,T)$ is the
product of an element in some $\Ext^{n-1}_{\Ag}(\as{S}_i,T)$ with an
element in $\Ext^1_{\Ag}(\as{S}_0,\as{S}_i)$, or is the product of an element in
some $\Ext^{n-2}_{\Ag}(\as{S}_i,T)$ with an element in
$\Ext^2_{\Ag}(\as{S}_0,\as{S}_i)$.  This completes the proof.
\end{proof}

We end this section with an immediate application to $\cK_2$ algebras. The concept of a $\cK_2$ algebra was introduced by Cassidy and Shelton in \cite[Definition 1.1]{CS}, where they defined a graded algebra $A$ to be $\cK_2$ if the Ext algebra, $E(A)$, is generated as an algebra in degrees 0, 1 and 2. This class of algebras is a natural generalization of the class of Koszul algebras.

\begin{cor}\label{cor:K2-fg-ext}
Let $K$ be an algebraically closed field and let $(\Gamma,\mfo,\mm,\mq)$ be a quantized Brauer graph with no truncated edges. Let $\Ag$ denote the associated Brauer graph algebra. Suppose $\Ag$ is length graded. Then $\Ag$ is a ${\mathcal K}_2$ algebra.
\end{cor}

\begin{proof}
From the discussion at the beginning of this section, we may assume that $(\Gamma,\mfo,\mm, \mq)$ is a quantized Brauer graph with $\mm\equiv 1$, with no loops or multiple edges and no truncated edges. From Proposition~\ref{prop:reduction of quantizing function}, we may assume further that $\mq \equiv 1$ since the number of generators of the Ext algebra, $E(\Ag)$, and their degrees do not depend on $\mq$. The result now follows from Theorem~\ref{thm:fg-ext} and the fact that $\Ag$ is a graded algebra.
\end{proof}

We consider ${\mathcal K}_2$ algebras and other generalizations of Koszul algebras further in the next section.

\section{Length graded Brauer graph algebras}\label{sec:truncated_edge}

In this section, $(\Gamma,\mathfrak o,\mm)$ is a Brauer graph and $\cA_{\Gamma}$ denotes the associated Brauer graph algebra.

Our first results lead to Theorem~\ref{thm:notK2}, where we characterise the Brauer graph algebras $\cA_{\Gamma}$ where the Ext algebra, $E(\cA_{\Gamma})$, is finitely generated in degrees 0, 1 and 2. This provides a converse to Theorem~\ref{thm:fg-ext}. In the remainder of the section we consider length graded Brauer graph algebras $\cA_{\Gamma}$. We recall the definition from \cite{GM} of a 2-$d$-Koszul algebra. We then complete our study of generalizations of Koszul algebras with Theorem~\ref{thm:2-d-Koszul}, where we classify the Brauer graph algebras that are 2-$d$-Koszul. Indeed, Theorem~\ref{thm:2-d-Koszul} shows that a Brauer graph algebra is a 2-$d$-Koszul algebra if and only if it is 2-$d$-homogeneous and a ${\mathcal K}_2$ algebra.

We begin with a result on syzygies of string modules where the Brauer graph $\Gamma$ may have both truncated and nontruncated edges.

\begin{prop}\label{prop:first-sy-truncated}
Let $(\Gamma,\mathfrak o,\mm)$ with $\mm\equiv 1$ be a
Brauer graph with no loops or multiple edges, and let
$\cA_{\Gamma}$ denote the associated Brauer graph algebra.
Let $\sigma=\as{s}_1^{e_1},\dots, \as{s}_{n}^{e_{n}}$ be an acceptable sequence
of weighted edges in $\Gamma$, and $n\geqslant 2$. Let $M=\st(\sigma)$.
\begin{enumerate}[\bf (1)]
\item Suppose that $n=2$. Then without loss of generality, $M$ is the uniserial module $\st(\as{s}_1^+,\as{s}_2^-)$.  Let $\alpha$ be the vertex at which $\as{s}_1$ and $\as{s}_2$ are both incident and let $\as{u}$ be the successor of $\hat{s}_2$ at $\alpha.$
\begin{enumerate}[$\bullet$]
\item If $\as{s}_1$ is a nontruncated edge, then the first syzygy of $M$ is isomorphic to the string module $\st(\as{u}^+,\as{s}_1^-,\as{t}^+)$ if $\as{u}\neq \as{s}_1$, and to the string module $\st(\as{s}_1^-,\as{t}^+)$ if $\as{u}= \as{s}_1$, where $\as{t}$ follows $(\as{s}_2,\as{s}_1)$.
\item If $\as{s}_1$ is a truncated edge, then the
first syzygy of $M$ is isomorphic to the string module $\st(\as{u}^+,\as{s}_1^-)$ if $\as{u}\neq \as{s}_1$, and to the simple module $\st(\as{s}_1^+)$  if $\as{u}= \as{s}_1$.
\end{enumerate}
\item Suppose that $n\pgq 3$ and $e_1 = +$.
\begin{enumerate}[$\bullet$]
\item If $\as{s}_1$ is a nontruncated edge, then the
first syzygy of $M$ is isomorphic to the string module $\st(\tau)$,
where $\tau$ begins as in Proposition~\ref{prop:first-sy}.
\item If $\as{s}_1$ is a truncated edge, then the
first syzygy of $M$ is isomorphic to the string module $\st(\tau)$,
where $\tau$ begins $\tau = \as{s}_1^{e_1^*},\as{s}_2^{e_2^*},\dots$
\end{enumerate}
\item Suppose that $n\pgq 3$ and $e_1 = -$.
\begin{enumerate}[$\bullet$]
\item If $\as{s}_1$ and $\as{s}_2$ are incident at the vertex $\alpha$ and if $\as{s}_2$ is the successor of $\as{s}_1$ at $\alpha$, then the
first syzygy of $M$ is isomorphic to the string module $\st(\tau)$,
where $\tau$ begins $\tau = \as{s}_2^{e_2^*},\as{s}_3^{e_3^*},\dots$
\item If $\as{s}_1$ and $\as{s}_2$ are incident at the vertex $\alpha$ and if $\as{s}_2$ is not the successor of $\as{s}_1$ at $\alpha$, then the first syzygy of $M$ is isomorphic to the string module $\st(\tau)$,
where $\tau$ begins $\tau = \as{t}^{+},\as{s}_2^{e_2^*},\as{s}_3^{e_3^*},\dots,$
and where $\as{t}$ is the successor of $\as{s}_1$ at $\alpha$.
\end{enumerate}
\item The first syzygy of $M$ may be fully determined using {\bf(1)}, {\bf(2)}, {\bf(3)} and Corollary~\ref{cor:string-rev}.
\end{enumerate}
\end{prop}

The proof is similar to that of Proposition~\ref{prop:first-sy}; note that these syzygies also appear in \cite{AG} where they consider the Ext algebra of a symmetric Brauer graph algebra.

For a string module $\st(\sigma)$, where $\sigma=\as{s}_1^{e_1},\dots, \as{s}_{n}^{e_{n}}$ is an acceptable sequence of weighted edges in $\Gamma$, we see that the only $\as{s}_i$ that can be truncated edges are $\as{s}_1$ and $\as{s}_n$.
We may use Proposition~\ref{prop:first-sy-truncated} to find all syzygies of the simple $\Ag$-modules.
In the case where $S$ is the simple $\Ag$-module corresponding to a truncated edge $s$ in $\Gamma$, we can simplify
Proposition~\ref{prop:first-sy-truncated}, and have the following corollary.

\begin{cor}\label{cor:truncated_edge}
Let $(\Gamma,\mathfrak o,\mm)$ with $\mm\equiv 1$ be a Brauer graph with no
loops or multiple edges, and assume that $\Gamma \neq \aa_2$. Let $\cA_{\Gamma}$ denote the associated
Brauer graph algebra.  Suppose that edge $s$ in $\Gamma$ is truncated at vertex $\alpha$. Define edges $\as{s}_i$ recursively, so that $\as{s}_0 = s$, $\as{s}_1$ is the successor of $s$ at the endpoint $\beta \neq \alpha$, and, for $m\pgq 1$, if $\as{s}_{m}$ is not truncated at the vertex which is not incident with $\as{s}_{m-1}$, then $\as{s}_{m+1}$ follows $(\as{s}_{m-1},\as{s}_{m})$.
Then,
\begin{enumerate}[\bf (1)]
\item there is some $n \pgq 1$ so that $\as{s}_n$ is a truncated edge in $\Gamma$;
\item for $1\ppq m \ppq n$, the $m$-th syzygy of $S$ is $\st(\as{s}_m^+, \as{s}_{m-1}^-)$;
\item the $(n+1)$-st syzygy of $S$ is $\st(\as{s}_n^+)$.
\end{enumerate}
Moreover, $S$ is a periodic $\Ag$-module.
\end{cor}

The sequence $(\as{s}_0, \as{s}_1, \dots , \as{s}_n)$ of Corollary~\ref{cor:truncated_edge} is precisely the Green walk for $s = \as{s}_0$; see \cite{KR} and \cite{JAGreen}. We note that it is also known from \cite[Corollary 2.7]{R} that $S$ is a periodic $\Ag$-module when $s$ is a truncated edge in $\Gamma$.

For the next result, we recall that $\underline{\Hom}_{\Ag}(-,-)$ denotes the $\Ag$-module homomorphisms modulo those homomorphisms which factor through a projective $\Ag$-module.

\begin{thm}\label{thm:elt_in high_degree}
Let $(\Gamma,\mathfrak o,\mm)$ with $\mm\equiv 1$ be a Brauer graph with no
loops or multiple edges. Let $\cA_{\Gamma}$ denote the associated
Brauer graph algebra.  If $\Gamma$ has both truncated and
nontruncated edges, then the Ext algebra, $E(\Ag)$, is not generated in degrees 0, 1 and 2 alone.

In particular, if we have a sequence of edges $\as{s}_0, \as{s}_1,\dots, \as{s}_n$ with $n\pgq 2$, where
\begin{enumerate}[(i)]
\item $\as{s}_0$ and $\as{s}_n$ are truncated edges,
\item $\as{s}_1, \as{s}_2,\dots,\as{s}_{n-1}$ are nontruncated edges,
\item $\as{s}_1$ is the successor of $\as{s}_0$,
\item $\as{s}_{i+1}$ follows $(\as{s}_{i-1},\as{s}_{i})$ for $1 \ppq i \ppq n-1$,
\end{enumerate}
then there is an element of $\Ext_{\cA_{\Gamma}}^{n+1}(\as{S}_0,\as{S}_n)$ which is not in the subalgebra
of $E(\Ag)$ generated by the elements of degree at most $n.$
\end{thm}

\begin{proof}
Note that we continue to assume that the Brauer graph $\Gamma$ is connected. We shall use
Corollary~\ref{cor:truncated_edge} repeatedly without comment. Recall that
$\Ext_{\Ag}^k(\as{S}_0,T)\cong\Hom_{\Ag}(\Omega^k(\as{S}_0),T)\cong\underline{\Hom}_{\Ag}(\Omega^{k+\ell}(\as{S}_0),\Omega^\ell(T))$ for any simple module  $T$, and that for simple modules $S$, $T$ and $U,$ and integers $k$ and $\ell, $ the Yoneda product $\Ext_{\Ag}^\ell(T,U) \times \Ext_{\Ag}^k(S,T)
\rightarrow \Ext_{\Ag}^{k+\ell}(S,U)$ can be identified with the composition of maps
  $\underline{\Hom}_{\Ag}(\Omega^{k+\ell}(S),\Omega^\ell(T))\times \Hom_{\Ag}(\Omega^\ell(T),U)\rightarrow
  \Hom_{\Ag}(\Omega^{k+\ell}(S),U).$

Since $\as{s}_0$ is a truncated edge, we know that $\Ext^{n+1}_{\Ag}(\as{S}_0,\as{S}_n)\cong
\Hom_{\Ag}(\Omega^{n+1}(\as{S}_0),\as{S}_n)\cong\Hom_{\Ag}(\as{S}_n,\as{S}_n)\cong K$. We must prove that
homomorphisms from  $\as{S}_n$ to itself cannot be written as sums of Yoneda products of elements of
$\Ext_{\Ag}^i(\as{S}_0,T)$ and $\Ext_{\Ag}^j(T,\as{S}_n)$ with $T$ a simple $\Ag$-module,  $i+j=n+1$ and $i$ and $j$ nonzero.

First note that $\Ext^i_{\Ag}(\as{S}_0,T)\cong\Hom_{\Ag}(\Omega^i(\as{S}_0),T)\cong
\Hom_{\Ag}(\Top(\Omega^i(\as{S}_0)),T)$, and  we have
$\Top(\Omega^i(\as{S}_0))=\as{S}_i$  for some $i$ with $1\ppq i\ppq  n$. Therefore $T=\as{S}_i.$

Moreover,
$\Ext_{\Ag}^i(\as{S}_0,\as{S}_i)\cong\underline{\Hom}_{\Ag}(\Omega^{i+j}(\as{S}_0),\Omega^j(\as{S}_i))=\underline{\Hom}_{\Ag}(\Omega^{n+1}(\as{S}_0),\Omega^j(\as{S}_i))\cong\Hom_{\Ag}(\as{S}_n,\Omega^j(\as{S}_i))$
and $\Ext_{\Ag}^j(\as{S}_i,\as{S}_n)\cong
\Hom_{\Ag}(\Omega^j(\as{S}_i),\as{S}_n)$. These
two spaces must be nonzero; therefore $\as{S}_n$ must be in the socle and in the top of
$\Omega^j(\as{S}_i)$. Since $\as{s}_n$ is truncated, this means that $\Omega^j(\as{S}_i)$ is a mixed
string module, of the form $\st(\as{s}_n^-,\ldots,\as{s}_n^+)$, by Proposition \ref{prop:first-sy-truncated}. But then the composition
$\as{S}_n\rightarrow  \Omega^j(\as{S}_i)\rightarrow \as{S}_n$ must be zero, unless
$\Omega^j(\as{S}_i)=\as{S}_n$, since the first map must go into the socle and the second map comes
from the top.
Hence assume that $\Omega^j(\as{S}_i)=\as{S}_n$. Since $\as{S}_n=\Omega^{n+1}(\as{S}_0)$ and $\Ag$
is selfinjective, this implies that
$\as{S}_i=\Omega^{n+1-j}(\as{S}_0)=\Omega^{i}(\as{S}_0)$ so that
$\Omega^{i}(\as{S}_0)$ is simple. However,  $\Omega^{i}(\as{S}_0)$ is the nonsimple uniserial
module with top $\as{S}_i$ and socle $\as{S}_{i-1}$, so that we have a contradiction. Therefore any
Yoneda product $\Ext_{\Ag}^{n+1-i}(T,\as{S}_n)\times\Ext^i_{\Ag}(\as{S}_0,T)\rightarrow
\Ext^{n+1}_{\Ag}(\as{S}_0,\as{S}_n)$ with $1\ppq i\ppq n$ and $T$ simple is zero, and we have the
required result.
\end{proof}

The next theorem characterizes the Brauer graph algebras
$\cA_{\Gamma}$ where the Ext algebra, $E(\cA_{\Gamma})$, is finitely generated in degrees 0, 1 and 2, providing a converse to Theorem~\ref{thm:fg-ext}.

\begin{thm}\label{thm:notK2}
Let $(\Gamma,\mathfrak o,\mm)$ be a Brauer graph and let $\cA_{\Gamma}$ denote the associated
Brauer graph algebra. Then the Ext algebra, $E(\cA_{\Gamma})$, is finitely generated in degrees 0, 1 and 2 if and only if $\Gamma$ does not have both truncated and nontruncated edges.
\end{thm}

\begin{proof}
If $\Gamma$ has no truncated edges, then it follows from Theorem~\ref{thm:fg-ext} that $E(\Ag)$ is finitely generated in degrees 0,1 and 2. So suppose that $\Gamma$ has at least one truncated edge. If all the edges are truncated then $\Gamma$ is a star (including the case $\Gamma = \aa_2$) and the associated Brauer graph algebra $\Ag$ is a Nakayama algebra. It is well-known that such an algebra is $d$-Koszul (for $d \pgq 2$) and hence, by \cite{GMMVZ}, its Ext algebra is generated in degrees (at most) 0, 1 and 2. (Indeed, a $d$-Koszul algebra is also length graded and so is ${\mathcal K}_2$.)

Thus we may assume that $\Gamma$ has both truncated and nontruncated edges. From the discussion at the start of Section~\ref{sec:coh-ring} and Proposition~\ref{prop:ext-struct}, we may assume that $(\Gamma,\mfo,\mm)$ is a Brauer graph with $\mm\equiv 1$, with no loops or multiple edges and which also has both truncated and nontruncated edges. It is now immediate from Theorem~\ref{thm:elt_in high_degree}, that $E(\Ag)$ cannot be generated only in degrees 0, 1 and 2.
\end{proof}

\begin{cor}\label{cor:notK2}
Let $(\Gamma,\mathfrak o,\mm)$ be a Brauer graph and let $\cA_{\Gamma}$ denote the associated
Brauer graph algebra. Suppose $\cA_{\Gamma}$ is length graded. Then $\cA_{\Gamma}$ is ${\mathcal K}_2$ if and only if $\Gamma$ does not have both truncated and nontruncated edges.
\end{cor}

We now introduce 2-$d$-Koszul algebras, a class of graded algebras which includes the Koszul algebras. Recall from Section~\ref{sec:Koszul} that an algebra $\Lambda = K{\mathcal Q}/I$ is a 2-$d$-homogeneous algebra if $I$ can be generated by homogeneous elements of lengths 2 and $d$.

Let $\Lambda = K{\mathcal Q}/I$ where $I$ is generated by homogeneous elements, so that $\Lambda$ is length graded with $\Lambda = \Lambda_0\oplus\Lambda_1\oplus\Lambda_2\oplus\cdots$.
Following Green and Marcos in \cite{GM}, and for a function $F\colon{\mathbb N} \to {\mathbb N}$, the algebra $\Lambda$ is said to be {\it $F$-determined} (respectively, {\it weakly $F$-determined}) if the $n$-th projective module in a minimal graded projective resolution of $\Lambda_0$ (viewed as a graded $\Lambda$-module in degree 0) can be generated in degree $F(n)$ (respectively, $\ppq F(n)$), for all $n \in {\mathbb N}$. Let $\delta\colon{\mathbb N} \to {\mathbb N}$ be the map given by
\[\delta(n) = \begin{cases}
\frac{n}{2}d & \text{if $n$ is even}\\
\frac{n-1}{2}d + 1 & \text{if $n$ is odd.}
\end{cases}\]
An algebra $\Lambda = K{\mathcal Q}/I$ is then said to be a {\it 2-$d$-determined} algebra if $I$ can be generated by homogeneous elements of degrees $2$ and $d$, and if $\Lambda$ is weakly $\delta$-determined. Thus a 2-$d$-determined algebra is a 2-$d$-homogeneous algebra. Furthermore, a 2-$d$-determined algebra is said to be {\it 2-$d$-Koszul} if its Ext algebra is finitely generated.

Assume that $(\Gamma,\mathfrak o,\mm,\mq)$ is a quantized Brauer graph and let
$\cA_{\Gamma}$ denote the associated Brauer graph algebra.
Suppose that $\Ag$ is graded with the length grading. Then the conclusions of Theorems~\ref{thm:fg-ext} and \ref{thm:notK2} are still true for $\Ag$, provided the
field $K$ satisfies one of the two conditions in Corollary~\ref{cor:reduction of quantizing function}. This is the case if $K$ is algebraically closed, for which see Proposition~\ref{prop:reduction of quantizing function} and also Corollary~\ref{cor:K2-fg-ext}.
This allows us to classify the 2-$d$-homogeneous Brauer graph algebras which are 2-$d$-determined.

\begin{thm}\label{thm:2-d-Koszul}
Let $(\Gamma,\mathfrak o,\mm,\mq)$ be a quantized Brauer graph, let $\cA_{\Gamma}$ denote the associated
Brauer graph algebra, and assume that either $\mq\equiv1 $ or  the field $K$ satisfies the conditions in Corollary~\ref{cor:reduction of quantizing function}. Let $d \pgq 3$ and suppose that $\cA_{\Gamma}$ is 2-$d$-homogeneous. Then the following are equivalent:
\begin{enumerate}[{\bf (1)}]
\item $\Gamma$ has no truncated edges,
\item $\cA_{\Gamma}$ is 2-$d$-determined,
\item $\cA_{\Gamma}$ is 2-$d$-Koszul,
\item $\cA_{\Gamma}$ is $\cK_2$.
\end{enumerate}
\end{thm}

\begin{proof} The equivalence between {\bf(1)} and {\bf(4)} follows from Corollary~\ref{cor:notK2},
since a $2$-$d$-homogeneous Brauer graph algebra cannot have only truncated edges. By definition, {\bf(3)} implies {\bf(2)}. We prove that {\bf(1)} and {\bf(2)} are equivalent, and then that {\bf(2)} implies {\bf(3)}. Suppose throughout that $\Ag$ is $2$-$d$-homogeneous.

{\bf(2)} $\Rightarrow$ {\bf(1)}. Suppose that $\Gamma$ has at least one truncated edge $s$. By  Proposition
\ref{prop:2dhomogeneous}, we know that no two successors are truncated and that for any vertex
$\alpha$ in $\Gamma$ we have $\mm(\alpha)\val(\alpha)\in\set{1,d}.$ There are edges
$\as{s}_0=s,\as{s}_1,\ldots,\as{s}_{n-1},\as{s}_n$ such that $\as{s}_n$ is truncated, $n\pgq2$,
$\as{s}_1,\ldots,\as{s}_{n-1}$ are not truncated, $\as{s}_1$ is the successor of $\as{s}_0$ at the
vertex $\alpha_1$ and for
each integer $m$ with $1\ppq m<n$ the edge $\as{s}_{m+1}$ follows $(\as{s}_{m-1},\as{s}_m)$ at the
vertex $\alpha_{m+1},$ as in Corollary \ref{cor:truncated_edge}. Note that when $\as{s}_m$ is a
loop, it occurs twice in the successor sequence of $\as{s}_{m-1}$; when we say that
 $\as{s}_{m+1}$ follows $(\as{s}_{m-1},\as{s}_m)$ in this case, we have considered one instance of
 the loop being in the successor sequence of $\as{s}_{m-1}$, and  $\as{s}_{m+1}$ is the successor of
 $\as{s}_{m-1}$ at the other instance of this loop  in the successor sequence of $\as{s}_{m-1}$. We
 shall now prove that the third projective in a minimal projective resolution of the simple
 $\Ag$-module $\as{S}_{n-1}$ has a generator in degree $d+2>\delta(3)$ so that $\Ag$ cannot be
 $2$-$d$-determined.

Let $t$ be the successor of $\as{s}_{n-1}$ at $\alpha_{n-1}$ and, if $t$ is not truncated, let $u$
follow $(\as{s}_{n-1},t).$ Then the indecomposable projective module $P_{\as{S}_n}$ is the uniserial
module of length $d$ whose top and socle are $\as{S}_n$, which we represent by
$P_{\as{S}_n}=\vcenter{\tiny\xymatrix@R=2pt@M=0pt{\as{S}_n\\Y\\\as{S}_{n-1}\\\as{S}_n}}$. The
indecomposable projective module $P_{\as{S}_{n-1}}$ is biserial, we represent it by
$P_{\as{S}_{n-1}}=\vcenter{\tiny\xymatrix@R=3pt@C=2pt@M=0pt{&\as{S}_{n-1}\ar@{-}[dl]\\T&&\as{S}_{n}\ar@{-}[ul]\\X\ar@{-}[dr]&&Y\\&\as{S}_{n-1}\ar@{-}[ur]}}$,
where $\vcenter{\tiny\xymatrix@R=2pt@M=0pt{T\\X\\\as{S}_{n-1}}}$ is the uniserial module of
length $d-1$ with top $T$ and socle $\as{S}_{n-1}$ (defined from the successor sequence of $t$ at $\alpha_{n-1}$). We
shall also need $P_T$ and $P_U$ which we represent as follows:
\begin{xalignat*}{2}
&P_T=
\begin{cases}
\vcenter{\tiny\xymatrix@R=2pt@C=2pt@M=0pt{T\\X\\\as{S}_{n-1}\\T}} &\text{if $t$ is truncated}\\\\
\vcenter{\tiny\xymatrix@R=2pt@C=5pt@M=0pt{&T\ar@{-}[dr]\\U\ar@{-}[ur]&&X\\Z&&\as{S}_{n-1}\\&T\ar@{-}[ul]\ar@{-}[ur]}} &\text{if $t$ is not truncated}
\end{cases}
&\text{ and }&P_U=
\begin{cases}
\vcenter{\tiny\xymatrix@R=2pt@C=2pt@M=0pt{U\\Z\\T\\U} }   &\text{if $u$ is truncated}\\\\
\vcenter{\tiny\xymatrix@R=2pt@C=5pt@M=0pt{&U\ar@{-}[dl]\ar@{-}[dr]\\Z&&L\\T&&\\&U\ar@{-}[ul]\ar@{-}[uur]} }  &\text{if $u$ is not truncated,}
\end{cases}
\end{xalignat*}
where $U$, $Z$ and $L$ are uniserial modules defined as before from the appropriate successor
sequences.

Let $(Q^\bullet,f^\bullet)$ be a minimal projective resolution of $\as{S}_{n-1}$. Then $Q^1$ is
generated in degree $1$ (by arrows in the quiver). If the edge $t$ is truncated, then
$\Omega^2(\as{S}_{n-1})=\vcenter{\tiny\xymatrix@R=3pt@C=2pt@M=0pt{&\as{S}_{n-1}\ar@{-}[dl]\ar@{-}[dr]\\T&&\as{S}_n}}$, the
module $Q^2$ is generated in degree $d$, the next syzygy
is $\Omega^3(\as{S}_{n-1})=\vcenter{\tiny\xymatrix@R=3pt@C=2pt@M=0pt{X&&Y\\&\as{S}_{n-1}\ar@{-}[ul]\ar@{-}[ur]}}$ and
$Q^3$ is generated in degree $d+2>\delta(3).$
 If the edge $t$ not is truncated, then
$\Omega^2(\as{S}_{n-1})=\vcenter{\tiny\xymatrix@R=1pt@C=5pt@M=0pt{U&\\Z&&\as{S}_{n-1}\\&T\ar@{-}[ul]\ar@{-}[ur]&&\as{S}_n\ar@{-}[ul]}}$, and
the
module $Q^2$ is generated in degrees $2$ and $d$. If the edge $u$ is truncated, the next syzygy
is $\Omega^3(\as{S}_{n-1})=\vcenter{\tiny\xymatrix@R=4pt@C=3pt@M=0pt{&T\ar@{-}[dl]\ar@{-}[dr]\\U&&X&&Y\\&&&\as{S}_{n-1}\ar@{-}[ul]\ar@{-}[ur]}}$ and
$Q^3$ is generated in degrees $d+1$ and $d+2>\delta(3).$ Finally, if  the edge $u$ is not truncated, the next syzygy
is $\Omega^3(\as{S}_{n-1})=\vcenter{\tiny\xymatrix@R=4pt@C=3pt@M=0pt{L\ar@{-}[dr]&&T\ar@{-}[dl]\ar@{-}[dr]\\&U&&X&&Y\\&&&&\as{S}_{n-1}\ar@{-}[ul]\ar@{-}[ur]}}$ and
$Q^3$ is generated in degrees $3$, $d+1$ and $d+2>\delta(3).$

Therefore, if $\Gamma$ has a truncated edge, then $\Ag$ is not $2$-$d$-determined. Hence {\bf(2)} implies {\bf(1)}.

{\bf(1)} $\Rightarrow$ {\bf(2)}. Now assume that $\Gamma$ does not have any truncated edges. Then by  Proposition
\ref{prop:2dhomogeneous}, we know that for any vertex
$\alpha$ in $\Gamma$ we have $\mm(\alpha)\val(\alpha)=d.$ To prove that $\Ag$ is $2$-$d$-determined,
we shall follow the lines of  Proposition \ref{prop:projQ} and give a projective resolution of the
simple module $\as{S}_0$, in which the projectives will be generated in appropriate degrees. We need
to be more precise in our notation, since we are no longer assuming that there are no multiple
edges; in particular, in the notation $a(s,t,\alpha)$ defined in Section~\ref{sec:coverings}, the vertex $\alpha$ must be specified.

Let $\as{s}_0$ be an edge in $\Gamma$, and let $\alpha_0$ and $\beta_0$ be its endpoints.  Let
$\as{s}_1$ be the successor of $\as{s}_0$ at $\alpha_0$ and let $\alpha_1$ be the other endpoint of $\as{s}_1.$  Let
$\as{s}_{-1}$ be the successor of $\as{s}_0$ at $\beta_0$ and let $\beta_1$ be the other endpoint of
$\as{s}_{-1}.$ For an integer $n\pgq 2$ and for any integer $i$ with $2\ppq i\ppq n,$ let $\as{s}_i$
be the the edge that follows $(\as{s}_{i-2},\as{s}_{i-1})$  at $\alpha_{i-1}$, let $\alpha_i$ be the other endpoint of $\as{s}_i,$ let $\as{s}_{-i}$
be the the edge that follows $(\as{s}_{-i+2},\as{s}_{-i+1})$  at $\beta_{i-1}$, let $\beta_i$ be the other endpoint of $\as{s}_{-i}.$
Then the projectives $Q^n$ in a minimal projective resolution of $\as{S}_0$ are as in Proposition~\ref{prop:projQ}.  The map $f^1:Q^1\rightarrow Q^0$ is given by the matrix
\[ (a(\as{s}_0,\as{s}_{-1},\beta_0)\quad a(\as{s}_0,\as{s}_1,\alpha_0)). \] In order to define the
maps $f^n:Q^n\rightarrow Q^{n-1}$ for $n\pgq 2$, we need to define,
for $0\ppq j\ppq \lfloor \frac{n-1}{2}\rfloor,$  the following paths of length $d-1:$
\begin{align*}
\tilde{p}(\as{s}_{n-2j-1},\as{s}_{n-2j-2},\alpha_{n-2j-2})&:=p(\as{s}_{n-2j-1},\as{s}_{n-2j-2},\alpha_{n-2j-2})C_{\as{s}_{n-2j-2},\alpha_{n-2j-2}}^{m(\alpha_{n-2j-2})-1}\quad\text{
and}\\
\tilde{p}(\as{s}_{-n+2j+1},\as{s}_{-n+2j+2},\beta_{n-2j-2})&:=p(\as{s}_{-n+2j+1},\as{s}_{-n+2j+2},\beta_{-n+2j+2})C_{\as{s}_{-n+2j+2},\beta_{n-2j-2}}^{m(\beta_{n-2j-2})-1}.
\end{align*}
If $n\pgq 2$ is even and $0\ppq j\ppq \frac{n}{2}-1$, the matrix of $f^n:Q^n\rightarrow Q^{n-1}$ is
described as follows,
\[
\begin{cases}
\text{its $(j,j)$-entry is } -a(\as{s}_{-n+2j+1},\as{s}_{-n+2j},\beta_{n-2j-1})\\
\text{its $(j,j+1)$-entry is } \tilde{p}(\as{s}_{-n+2j+1},\as{s}_{-n+2j+2},\beta_{n-2j-2})\\
\text{its $(n-1-j,n-1-j)$-entry is } -\tilde{p}(\as{s}_{n-2j-1},\as{s}_{n-2j-2},\alpha_{n-2j-2})\\
\text{its $(n-1-j,n-j)$-entry is } a(\as{s}_{n-2j-1},\as{s}_{n-2j},\alpha_{n-2j-1}).
\end{cases}
 \]
If $n\pgq 3$ is odd and $0\ppq j\ppq \frac{n-3}{2}$, the matrix of $f^n:Q^n\rightarrow Q^{n-1}$ is
described as follows,
\[
\begin{cases}
\text{its $(j,j)$-entry is } a(\as{s}_{-n+2j+1},\as{s}_{-n+2j},\beta_{n-2j-1})\\
\text{its $(j,j+1)$-entry is } \tilde{p}(\as{s}_{-n+2j+1},\as{s}_{-n+2j+2},\beta_{n-2j-2})\\
\text{its $(n-1-j,n-1-j)$-entry is } \tilde{p}(\as{s}_{n-2j-1},\as{s}_{n-2j-2},\alpha_{n-2j-2})\\
\text{its $(n-1-j,n-j)$-entry is } a(\as{s}_{n-2j-1},\as{s}_{n-2j},\alpha_{n-2j-1})\\
\text{its $(\frac{n-1}{2},\frac{n-1}{2})$-entry is } a(\as{s}_{0},\as{s}_{-1},\beta_0)\\
\text{its $(\frac{n-1}{2},\frac{n+1}{2})$-entry is } a(\as{s}_{0},\as{s}_{1},\alpha_{0}).\\
\end{cases}
 \]
 The projective $Q^1$ is generated in degree $1$ (by arrows in the quiver $\Qg$), the projective
 $Q^2$ is generated in degrees $2$ and $d$ (by some elements in the set of minimal generators of the
 ideal $\Ig$, but this can also be seen directly). It can then be seen inductively, using the
 resolution given above, that $Q^n$ is generated in degrees at most $\delta(n)$; more precisely, the
 modules $P_{\as{S}_{n-2j}}$ and $P_{\as{S}_{-n+2j}}$ are generated in degrees $n+j(d-2)$ with $0\ppq
 j\ppq \lfloor \frac{n}{2}\rfloor.$ This is true of the resolution of any simple $\Ag$-module, and therefore $\Ag$ is $2$-$d$-determined. Thus {\bf(1)} implies {\bf(2)}, and we have that {\bf(1)}, {\bf(2)} and {\bf(4)} are equivalent.

{\bf(2)} $\Rightarrow$ {\bf(3)}.
Suppose $\Ag$ is $2$-$d$-determined. Then, by the equivalence of {\bf (2)} and {\bf(4)}, we know that $\Ag$ is ${\mathcal K}_2$. Hence the Ext algebra, $E(\Ag)$, is generated in degrees 0, 1 and 2 and so is finitely generated. Thus $\Ag$ is $2$-$d$-Koszul and hence {\bf(3)} holds. This completes the proof.
\end{proof}

This theorem gives a positive answer for Brauer graph algebras to all three questions asked by Green and Marcos in \cite[Section 5]{GM}. Specifically, we have shown, for a Brauer graph algebra $\Ag$ which is 2-$d$-homogeneous, that
\begin{enumerate}
\item[(1)] if $\Ag$ is a 2-$d$-determined algebra, then $E(\Ag)$ is finitely generated;
\item[(2)] if $\Ag$ is a 2-$d$-determined algebra and if $E(\Ag)$ is finitely generated, then $E(\Ag)$ is generated in degrees 0, 1 and 2;
\item[(3)] if $E(\Ag)$ is generated in degrees 0, 1 and 2, then $\Ag$ is a 2-$d$-determined algebra.
\end{enumerate}

In addition, note that algebras $\Lambda_N$ of \cite{ST1}, where $N \pgq 1$, are all Brauer graph algebras, where $\Gamma$ is the oriented cycle with every vertex having multiplicity $N$. Moreover, these algebras are 2-$2N$-homogeneous and, using Theorem~\ref{thm:2-d-Koszul}, we see that they are also 2-$2N$-Koszul. This example gives a new class of 2-$d$-Koszul algebras.

In contrast to Theorem~\ref{thm:2-d-Koszul}, a negative answer was given by Cassidy and Phan to the first two questions posed by Green and Marcos in \cite{GM}. In \cite{CP}, Cassidy and Phan give specific infinite-dimensional algebras $A$ and $B$ such that $A$ is 2-4-determined but $E(A)$ is not finitely generated, and $B$ is 2-4-determined of infinite global dimension, $E(B)$ is finitely generated, but $E(B)$ is not generated in degrees 0, 1 and 2. Another generalization of Koszul is given by Herscovich and Rey in \cite{HR}, where they study multi-Koszul algebras. In particular, they remark that a left $\{2,d\}$-multi-Koszul algebra is a 2-$d$-Koszul algebra in the sense of \cite{GM}, though the converse does not hold. However, they showed that the Ext algebra of a multi-Koszul algebra $A$ is generated in degrees 0, 1 and 2, so that the algebra $A$ itself is ${\mathcal K}_2$.

\end{document}